\documentclass [12pt]{article}
 \usepackage{amssymb}
 \usepackage{amsmath,amsthm}
 \usepackage{mathrsfs}
 \usepackage{color}
 \usepackage{graphicx}
  \usepackage{tikz}

\normalsize

\renewcommand{\dfrac}[2]{\lower0.15ex\hbox{\large$\textstyle\frac{#1}{#2}$}}

\newtheorem{theorem}{Theorem}[section]
\newtheorem{lemma}[theorem]{Lemma}

\newtheorem{corollary}[theorem]{Corollary}

\newtheorem{remark}[theorem]{Remark}

\numberwithin{equation}{section}

\begin{document}
%\linenumbers

\title
{Large induced distance matchings in certain sparse random
graphs
\thanks{The work was supported  by NSFC.}}%11871377, (No.12071274)+862165904070(Telephone Number)
\author
{{Fang Tian$^1$\thanks{Corresponding Author:\
tianf@mail.shufe.edu.cn(Email Address)}\quad Yu-Qin Sun$^2$\quad Zi-Long Liu$^3$}\\
{\small $^1$School of Mathematics}\\
{\small Shanghai University of Finance and Economics, Shanghai, 200433, China}\\
{\small\tt tianf@mail.shufe.edu.cn}\\
{\small $^2$School of Mathematics and Physics}\\
{\small Shanghai University of Electric Power, Shanghai, 200090, China}\\
{\small\tt 2008000011@shiep.edu.cn}\\
{\small $^3$School of Optical-Electrical and Computer Engineering}\\%
{\small University of Shanghai for Science and Technology, Shanghai,
200093, China} \\
{\small\tt liuzl@usst.edu.cn}\\[1ex]
}

\date{}
 \maketitle

{\bf Abstract}\quad
For a fixed integer $k\geqslant 2$,
let $G\in \mathcal{G}(n,p)$ be a simple connected graph on $n\rightarrow\infty$
vertices with the expected degree $d=np$ satisfying $d\geqslant c$ and $d^{k-1}= o(n)$
for some large enough constant $c$. We show that the asymptotical size of any maximal
 collection of edges $M$ in $G$ such that no two edges in $M$
are within distance $k$, which is called a distance $k$-matching, is between
$ \frac{(k-1)n\log d}{4d^{k-1}}$ and $ \frac{k n \log d}{2d^{k-1}}$.
We also design a randomized greedy algorithm
to generate one large distance $k$-matching in $G$ with asymptotical size
$ \frac{kn\log d}{4d^{k-1}}$.
 Our results partially generalize the results on the size  of
the largest distance $k$-matchings from the case $k=2$ or $d=c$ for some large constant $c$.

\vskip0.4cm \noindent {\bf Keywords}\quad random graphs, induced matchings, distance matchings,
expected degree.

\vskip0.4cm \noindent {\bf Mathematics Subject Classification}\quad  05C35, 05C70, 05C80.

\section{Introduction}

Random graphs are classical mathematical models to capture  the main behaviors
of real-world graphs in massive data sets. Recall that the  Erd\H{o}s-R\'{e}nyi
random graph model $\mathcal{G}(n,p)$  is on the vertex set $[n]=\{1,\cdots,n\}$
and every edge appears independently with probability~$p$. A classic question
is to investigate whether $G\in\mathcal{G}(n,p)$ contains a copy of given graph
as an induced subgraph with high probability (abbreviated to w.h.p., meaning
with probability tending to $1$ as $n$ tends to infinity).
Some large induced subgraphs to look for in $G\in\mathcal{G}(n,p)$ are independent sets,
trees,  cycles and matchings.
%Matching problems are closely related with capacity problems, such as resource allocation problems,
%channel assignment problems and ¡°risk-free¡± marriage problems.

For any fixed positive integer $k$, a set of edges ${M}_k$ of a graph $G=(V,E)$
defined on the vertex set $[n]$, with the additional constraint that
no two edges are within distance $k$, is called a distance $k$-matching.
The distance between two vertices is the number of edges in a shortest path
between them, while the distance between two edges is the number of
vertices in a shortest path between them. We often refer to a path by the natural
sequence of its vertices, writing $P=x_0x_1\cdots x_k$ with distinct vertices
$x_i$ for $0\leqslant i\leqslant k$ and calling $P$ a path between $x_0$ and $x_k$ of length $k$.
The qualifier ``distance" is normally omitted in the remainder of this paper.
A $1$-matching is a matching and a $2$-matching is also known as an induced matching
or a strong matching~\cite{oliver21,joo14,maft94}.
A $k$-matching ${M}_k$ is said to be maximal if no other edges can be added to ${M}_k$ while
keeping the property of being a distance $k$ away.
Let $e(M_k)$ denote the size of the $k$-matching $M_k$.
The $k$-matching number, denoted by $um_k(G)$, is $e(M_k)$ of the largest $k$-matching $M_k$
in $G$.

Little is known about the random variable $um_k(G)$ when  $G\in\mathcal{G}(n,p)$ and  $k\geqslant 2$.
Maftouhi and Gordones~\cite{maft94} firstly considered
$um_2(G)$ for  $G\in\mathcal{G}(n,p)$ when  $0<p<1$ is a constant.
Czygrinow and Nagle~\cite{andr04} generalized this result. They showed that w.h.p.
$um_2(G)\leqslant \frac{n\log d}{d}$ when the expected degree $d=np=\Omega(1)$
and $d=o(n)$, further $um_2(G)\leqslant \log_h n$ with $h= \frac{1}{1-p}$
when $p> n^{-\epsilon}$ for any given $\epsilon>0$.
In a specially  sparse case when  $d$ is a sufficiently large  constant,
for any fixed integer $k\geqslant 2$ and $G\in \mathcal{G}(n,p)$,
Kang and Manggala~\cite{kang12} improved the upper bound of $um_k(G)$
 to $um_k(G)\leqslant (1+o(1))\frac{k n\log d}{2d^{k-1}}$.
It is speculated in~\cite{kang12} that the upper bound here is close to the correct
  value of $um_k(G)$ when $k\geqslant 2$.

  Recently, Cooley et al. in~\cite{oliver21} confirmed
  this conclusion when $k=2$, $ d\geqslant{c}$ and $d=o(n)$ for some large enough constant $c$.
They showed that $um_2(G)=(1+o(1)) \frac{n\log d}{d}$ relying on two
main ingredients: the second moment method and Talagrand's inequality.
  Talagrand's inequality is a useful tool to show that a random
variable is tightly concentrated  under certain conditions.
In fact,
 they showed that $um_2(G)= (1+o(1))\log_h d$ with $h= \frac{1}{1-p}$ in a broader range of $p$.
%They regard $G\in \mathcal{G}(n,p)$ as
%a product of $n-1$ probability spaces $Z_i$ for $i\in[n-1]$, where each $Z_i$
%picks uniformly at random a subset of $[i]$ of size ${\rm Bin}(i,p)$ as the neighbours
%of the vertex $i+1$ within $[i]$.
Unfortunately, the conditions of Talagrand's inequality  only hold on $um_2(G)$.
It is interesting  to consider the lower bounds of $um_k(G)$ when $k\geqslant 3$, while
there are less discussions.
%For a constant $p$ and $G\in\mathcal{G}(n,p)$, Clark~\cite{clark01} showed that
% $um_2(G)\geqslant(1+o(1))\log_h d$ with $h= \frac{1}{1-p}$.
Tian in~\cite{tian18} showed a trivial lower bound  of $um_k(G)=\Omega(\log n)$
when $d$ is a sufficiently large constant.

In this paper, for any fixed integer $k\geqslant 2$ and $G\in \mathcal{G}(n,p)$,
as a further step of~\cite{oliver21,andr04,kang12,tian18}, we consider the
bounds of $um_k(G)$. In fact, we mainly
 investigate $e(M_k)$ of any
maximal $k$-matching $M_k$ in $G\in \mathcal{G}(n,p)$.
Let $d(n)=np$ be the expected degree of $G$,
and we usually write $d$ instead of $d(n)$.

\begin{theorem}
Let $k\geqslant2$ be a fixed integer, $G\in \mathcal{G}(n,p)$ with
$d\geqslant c$ and $d^{k-1}=o(n)$ for some large enough constant $c$.
For any maximal $k$-matching $M_k$ in $G$, with high probability, $(1+o(1)) \frac{(k-1)n\log d}{4d^{k-1}}\leqslant e(M_k)\leqslant (1+o(1)) \frac{k n \log d}{2d^{k-1}}$.
\end{theorem}

 As a corollary of Theorem 1.1, we actually obtain the bounds of $um_k(G)$
when  $k\geqslant 3$.
\begin{corollary}
Let $k\geqslant3$ be a fixed integer, $G\in \mathcal{G}(n,p)$ with
$d\geqslant c$ and $d^{k-1}=o(n)$ for some large enough constant $c$.
Then, with high probability, $(1+o(1)) \frac{(k-1)n\log d}{4d^{k-1}}\leqslant
um_k(G)\leqslant (1+o(1)) \frac{k n \log d}{2d^{k-1}}$.
\end{corollary}

 In fact, we can run  a randomized greedy
 algorithm to generate a specially large $k$-matching with asymptotic size
 $ \frac{k n\log d}{4d^{k-1}}$ for an input graph $G\in \mathcal{G}(n,p)$.
It improves the lower bound of $um_k(G)$
 in Corollary~1.2.

%\noindent{\textbf {Algorithm}}:\ The large $k$-matching generator\\
%{\bf Input}: A graph $G\in \mathcal{G}(n,p)$  with $n\rightarrow\infty$,
%$d\rightarrow\infty$, $k\geqslant 2$ and $d^{k-1}=o(n)$.\\
%{\bf Output}: A large $k$-matching of $G$.\\
%{\bf Step 1}: Uniformly at random choose $2s$ vertices $S$ from $[n]$.
%Pair them into $s$ pairs randomly and independently, where $s= \frac{n}{4d^{k-1}}[k\log d- 3\log(k\log d)]$.\\
%{\bf Step 2}:  If $S$ is not a $k$-matching in $G$, replace one pair by an edge
%in $V(G)\setminus S$ at distance at least $k$ to $S$.\\
%{\bf Step 3}: Output $S$ if $S$ is a $k$-matching in $G$. Otherwise, go to step 2.

\begin{theorem}
Let $k\geqslant2$ be a fixed integer,  $G\in \mathcal{G}(n,p)$ with
$d\geqslant c$ and $d^{k-1}=o(n)$ for some large enough constant $c$. With high probability,
there exists a $k$-matching with size $(1+o(1))\frac{k n \log d}{4d^{k-1}}.$
\end{theorem}

%\begin{theorem}
%Let $k\geqslant2$ be any given integer,  $G\in \mathcal{G}(n,p)$ with
%the expected degree $d=np$ and $n\rightarrow\infty$, where $d\rightarrow\infty$ when $n\rightarrow\infty$
%and $d^{k-1}=o(n)$. With high probability, the randomized greedy
% algorithm ``The large $k$-matching generator"
%obtains a $k$-matching with size $(1+o(1))\frac{k n \log d}{4d^{k-1}}.$
%\end{theorem}

We remark that $d^{k-1}=o(n)$ is a necessary condition in the proof
of Theorem~1.1 and Theorem 1.3. The way to obtain the
upper bound of  the largest $k$-matching in Theorem 1.1 is similar with the one when $d$ is a
 large constant in~\cite{kang12}. For $k=2$, it
 coincides with the upper bound $um_2(G)\leqslant (1+o(1))
 \frac{n\log d}{d}$  when $d=\Omega(1)$ and $d=o(n)$ in~\cite{oliver21,andr04}.
 The  way to obtain
 the lower bound of $um_k(G)$ for $k\geqslant 3$ must be %we have demonstrated that
 distinct with the one in~\cite{oliver21}. %since $um_k(G)$
% is a Lipschitz and $\theta$-certifiable random variable only when $k=2$.
 %We obtain the lowerin Theorem~1.1 and Theorem~1.3 are  different,
 %and both are also
 We firstly
show the lower bound of $e(M_k)$ for any maximal $k$-matching $M_k$  in
Theorem~1.1, and we further  investigate an interesting property for some given number of vertices
of $G$ to show the other one in
Theorem~1.3.  %\in \mathcal{G}(n,p)
 We believe it is hard to improve the results here. % by these two approaches.

The organization of our paper is as follows. Notation and  auxiliary
results used throughout the paper are presented in Section 2.
For $G\in \mathcal{G}(n,p)$, the upper and lower bounds of $e(M_k)$ for
any maximal $k$-matching $M_k$ in Theorem~1.1 %by the first moment method
are discussed in Section 3 and Section 4, respectively.
%In Section 5, we show the improved lower bound of $um_k(G)$ in Theorem~1.2.
In Section 5, we design a simple randomized greedy algorithm
 to generate a large $k$-matching $M_k$ in $G$ and prove its efficiency.
The last section concludes the works.

 \vskip20pt
\section{Preliminaries}

In this section, we introduce a few  useful lemmas that help us prove
the main results and fix our notations.
We will use the standard Landau notations $o(\cdot), O(\cdot), \Omega(\cdot)$ and $\Theta(\cdot)$.
Throughout the following sections we assume that $n\rightarrow\infty$,
$c$ is a large enough constant and $k\geqslant 2$ is a fixed integer.
When not otherwise explicitly stated, the asymptotics in this notation
are with respect to $n$ or with respect to $c$.
For example,  for two positive-valued functions $f, g$,
we write $f\sim g$ to denote $\lim_{n\rightarrow \infty}f/g=1$ or
 $\lim_{c\rightarrow \infty}f/g=1$, and $f=o(g)$ to denote $\lim_{n\rightarrow \infty}f/g=0$
 or $\lim_{c\rightarrow \infty}f/g=0$.
For an event $\mathcal{A}$ and a random variable $Z$ in an
arbitrary probability space $(\Omega,\mathcal{F},\mathbb{P})$,
$\overline{\mathcal{A}}$ denotes the complement of the event $\mathcal{A}$,
$\mathbb{P}[\mathcal{A}]$, $\mathbb{E}[Z]$ and $\mathbb{V}[Z]$  denote the probability of $\mathcal{A}$,
 the expectation and the variance
of $Z$.  An event is said to occur  with high probability (w.h.p.  for short), if
the probability that it holds tends to 1 when $n\rightarrow\infty$.
Recall that ${\rm Bin}(n,p)$ denotes a Binomially distributed random variable with parameters $n$ and $p$,
that is, it is the sum of $n$ independent variables, each equal to $1$ with probability $p$ and $0$ otherwise.
All logarithms are natural. The floor and ceiling signs are omitted whenever they are not crucial.

We use the following standard Chernoff inequality to estimate the
 tail probability of the binomial distribution,  and Janson's inequality to
 estimate the probability that none of a set of bad events occur when these
 events are mostly independent, see~\cite{alon08}.  The statements have
 been adapted slightly from~\cite{alon08} in terms of our settings.

\begin{lemma}  {\rm{(Chernoff Inequality)}} For  any $\delta\in (0,1)$,
\begin{align*}
\mathbb{P}[|{\rm Bin}(n,p)-np|>\delta np]<2\exp[-\delta^2np/2].
\end{align*}
\end{lemma}

Let $G\in\mathcal{G}(n,p)$ be a random edge subset of the complete graph $K_n$ and
$u,v$ be any two vertices in $[n]$. Let $P_1,P_2,\cdots,P_{l_1}$ denote all paths
of length at most $k-1$ in $K_n$ to connect the vertex $u$ with the vertex $v$. We have
only one edge in $K_n$ to connect $u$ with $v$.
While there are $(n-2)\cdots(n-i)$ paths of length  $i$ %\sum_{i=2}^{k-1}
to connect $u$ with $v$ when $2\leqslant i\leqslant k-1$ and $k\geqslant 3$.
In summary, we have
\begin{equation*}
\begin{aligned}[b]
l_1&=\begin{cases}1,&k=2,\\1+\sum_{i=2}^{k-1}(n-2)\cdots(n-i),&k\geqslant3.
\end{cases}
\end{aligned}
\end{equation*}
Let $\mathcal{A}_i$ be the event that the path  $P_i$ exists in $G$ for $i\in [l_1]$.
For $i,j\in [l_1]$, we write $i\sim j$ if $i\neq j$ and $P_i\cap P_j\neq \emptyset$.
Let $\Delta=\sum_{i\sim j}\mathbb{P}[\mathcal{A}_i\cap \mathcal{A}_j]$
 and $U=\prod_{i=1}^{l_1}(1-\mathbb{P}[\mathcal{A}_i])$.
\begin{lemma} {\rm{(Janson's Inequality)}}\
 If $\mathbb{P}(\mathcal{A}_i)\leqslant \frac{1}{2}$ for any  $i\in [l_1]$,
 then $U\leqslant\mathbb{P}[\cap_{i=1}^{l_1} \overline{\mathcal{A}_i}]\leqslant U\exp[\Delta]$.
\end{lemma}

Similarly, let ${M}$ be a matching with size $m$ in $K_n$. Assume that $uu'$ and $vv'$ are
two distinct edges in $M$, which have $\binom{m}{2}$ ways to choose them.
To count the number of paths of length at most $k-1$ in $K_n$ to connect the  edge $uu'$
with the edge $vv'$, there are four edges $uv$, $uv'$, $u'v$ or
 $u'v'$  in $K_n$ between the edges $uu'$ and $vv'$;
 $4(n-4)\cdots(n-i-2)$ counts the number of paths of  length $i$ in $K_n$ to connect the edge $uu'$
with the edge $vv'$ when $2\leqslant i\leqslant k-1$ and $k\geqslant 3$.
 Let $P'_1,P'_2,\cdots,P'_{l_2}$ be the paths of length at most $k-1$ in $K_n$
 between  any two edges of $M$. We  have
\begin{equation*}
\begin{aligned}[b]
l_2&=\begin{cases}4 \binom{m}{2},&k=2,\\4 \binom{m}{2}+4 \binom{m}{2}\sum_{i=2}^{k-1}(n-4)\cdots(n-i-2),&k\geqslant3.
\end{cases}
\end{aligned}
\end{equation*}
Let $\mathcal{B}_i$ be the event that the path $P'_i$ exists in $G$
for $i\in [l_2]$. For $i,j\in [l_2]$, we write $i\sim j$ if $i\neq j$
and $P'_i\cap P'_j\neq \emptyset$.
Let $\Delta^\prime=\sum_{i\sim j}\mathbb{P}[\mathcal{B}_i\cap \mathcal{B}_j]$
and  $U^\prime=\prod_{i=1}^{l_2}(1-\mathbb{P}[\mathcal{B}_i])$.
\begin{lemma} {\rm{(Janson's Inequality)}}\
If  $\mathbb{P}(\mathcal{B}_i)\leqslant \frac{1}{2}$ for any  $i\in [l_2]$,
 then $U^\prime\leqslant\mathbb{P}[\cap_{i=1}^{l_2} \overline{\mathcal{B}_i}]\leqslant
 U^\prime\exp[\Delta^\prime]$.
\end{lemma}

 We make use of these inequalities to show some probability inequalities of  distances in
 $G\in\mathcal{G}(n,p)$.
\begin{lemma}
Let $k\geqslant2$ be a fixed integer, $G\in \mathcal{G}(n,p)$, $d=np$ satisfying
$d\geqslant c$ and $d^{k-1}=o(n)$ for some large enough constant $c$. %\log d
Define $p_d$ to satisfy the equation $np_{d}=d^{k-1}$.

\noindent{\rm (a)}\ Given two vertices $u$ and $v$ in $[n]$, let $d_{G}(u,v)$
denote the distance between  $u$ and $v$ in $G$.
Then,
\begin{equation*}
\begin{aligned}[b]
\mathbb{P}[d_{G}(u,v)\geqslant k]&=\begin{cases}1-p,&k=2;\\(1-p_{d})
\exp[{O(n^{k-3}p^{k-2}+n^{2k-4}p^{2k-2})}],&k\geqslant3.
\end{cases}
\end{aligned}
\end{equation*}

\noindent{\rm (b)}\  Let ${M}$ be a matching with size $m$  in  $K_n$ and
$m=O\bigl( \frac{n\log d}{d^{k-1}}\bigr)$.
Then,
\begin{equation*}
\begin{aligned}[b]
&\mathbb{P}[{M} \ {\rm is\ a}\  k{\mbox-}{\rm matching\ in}\
G]\\&=\begin{cases}p^m(1-p)^{4 \binom{m}{2}},&k=2;\\p^m(1-p_{d})^{4 \binom{m}{2}}\exp[O(m^2n^{k-3}p^{k-2}+m^2n^{2k-4}p^{2k-2}+m^3n^{2k-5}p^{2k-3})],&k\geqslant3.
\end{cases}
\end{aligned}
\end{equation*}
\end{lemma}

\begin{proof} (a)\  For $k=2$, it is obviously true because $\mathbb{P}[d_{G}(u,v)\geqslant 2]=1-p$.
For $k\geqslant 3$, we consider $\Delta$ and $U$  in Lemma 2.2 because
$\mathbb{P}[d_{G}(u,v)\geqslant
k]=\mathbb{P}[\cap_{i=1}^{l_1} \overline{\mathcal{A}_i}]$.
Note that  $p=d/n=o(1)$ and $k\geqslant 3$ a fixed integer, then  we have
\begin{align}
U&=(1-p)\prod\limits_{i=2}^{k-1}(1-p^i)^{(n-2)\cdots(n-i)}\notag\\
&=(1-p)\prod\limits_{i=2}^{k-1}\Bigl[1-(n-2)\cdots(n-i)p^i+O\bigl(n^{2i-2}p^{2i}\bigr)\Bigr]\notag\\
&=\Bigl[1-p-\sum_{i=2}^{k-1}(n-2)\cdots(n-i)p^i\Bigr]\cdot\Bigl[1+O\bigl(n^{2k-4}p^{2k-2}\bigr)\Bigr],
\end{align}
where
\begin{align*}
(1-p^i)^{(n-2)\cdots(n-i)}=1-(n-2)\cdots(n-i)p^i+O\bigl(n^{2i-2}p^{2i}\bigr)
\end{align*}
is true because Taylor's expansion and $n^{i-1}p^i=o(1)$ when $d^{k-1}=o(n)$  and $2\leqslant i\leqslant k-1$;
the last equality is true because $n^{2k-4}p^{2k-2}$ is the largest term in $n^{2i-2}p^{2i}$
when $2\leqslant i\leqslant k-1$ and $d\geqslant c$ for some large enough constant $c$.

Since
\begin{align*}
p+\sum_{i=2}^{k-1}(n-2)\cdots(n-i)p^i=\sum_{i=1}^{k-1} n^{i-1}p^i+O\bigl(n^{k-3}p^{k-1}\bigr)
\end{align*}
and $n^{k-3}p^{k-1}=o(n^{2k-4}p^{2k-2})$, by the equation in (2.1), it follows that
\begin{align}
U
&=\bigl(1-\sum_{i=1}^{k-1} n^{i-1}p^i\bigr)\cdot\Bigl[1+O\bigl(n^{2k-4}p^{2k-2}\bigr)\Bigr]\notag\\
&=\bigl(1-\sum_{i=1}^{k-1} n^{i-1}p^i\bigr)\exp\bigl[O\bigl(n^{2k-4}p^{2k-2}\bigr)\bigr],
\end{align}
where
\begin{align*}
\exp\bigl[O\bigl(n^{2k-4}p^{2k-2}\bigr)\bigr]=1+O\bigl(n^{2k-4}p^{2k-2}\bigr)
\end{align*} is true because Taylor's expansion %approximate
and $n^{2k-4}p^{2k-2}=d^{2k-2}/n^2=o(1)$ when $d^{k-1}=o(n)$.

By $d=np$ and $d^{k-1}=np_{d}$, we have
\begin{align*}
1-\sum_{i=1}^{k-1} n^{i-1}p^i &= 1- \frac{p[1-(np)^{k-1}]}{1-np}=
1- \frac{p(1-d^{k-1})}{1-d}= 1+ \frac{\frac{1}{n}-p_d}{1- \frac{1}{d}}
\end{align*}
in (2.2). Note that
\begin{align*}
1+ \frac{\frac{1}{n}-p_d}{1- \frac{1}{d}}
&= 1 + \Bigl(\frac{1}{n}-p_d\Bigr)\Bigl(1+O\Bigl( \frac{1}{d}\Bigr)\Bigr)=1-p_{d}+O\bigl(n^{k-3}p^{k-2}\bigr)
\end{align*}
because $d\geqslant c$ for some large enough constant $c$ and $O(p_d/d)=O(n^{k-3}p^{k-2})$. Thus, we obtain
\begin{align}
1-\sum_{i=1}^{k-1} n^{i-1}p^i
&=(1-p_d)\exp\bigl[O\bigl(n^{k-3}p^{k-2}\bigr)\bigr].
\end{align}
By  the equations in (2.2) and (2.3),  we finally have
\begin{align}
U
&=\bigl(1-p_{d}\bigr)\exp\bigl[O\bigl(n^{k-3}p^{k-2}+n^{2k-4}p^{2k-2}\bigr)\bigr].
\end{align}

We need to bound $\Delta=\sum_{i\sim j}\mathbb{P}[\mathcal{A}_i\cap \mathcal{A}_j]$ in Lemma 2.2 below.
Assume that
the paths $P_i$ and $P_j$ between $u$ and $v$ are of lengths $\ell_i$ and $\ell_j$
satisfying $2\leqslant\ell_i\leqslant \ell_j\leqslant k-1$ and the size of common edges of them is denoted
by $t\geqslant 1$.  Note that $\Delta=\sum_{i\sim j}\mathbb{P}[\mathcal{A}_i\cap \mathcal{A}_j]=
\sum_{i\sim j}\mathbb{P}[\mathcal{A}_j| \mathcal{A}_i]\cdot\mathbb{P}[\mathcal{A}_i]$.
The ways to choose $P_i$ in $K_n$ are at most $n^{\ell_i-1}$ and $\mathbb{P}[\mathcal{A}_i]=p^{\ell_i}$.
Fix the path $P_i$, for a given set of
 $t$ edges in $P_i$,
then the ways to choose $P_j$ in $K_n$ are at most $n^{\ell_j-t-1}$ and $\mathbb{P}[\mathcal{A}_j| \mathcal{A}_i]=p^{\ell_j-t}$.
 Thus, we have
\begin{align}
\Delta=
\sum_{i\sim j}\mathbb{P}[\mathcal{A}_j| \mathcal{A}_i]\cdot\mathbb{P}[\mathcal{A}_i]&\leqslant \sum_{\ell_i=2}^{k-1}n^{\ell_i-1}p^{\ell_i}\sum_{\ell_j=\ell_i}^{k-1}
\sum_{t=1}^{\ell_i-1}\binom{\ell_i}{t}n^{\ell_j-t-1}p^{\ell_j-t}.%\\
\end{align}
Firstly, for fixed $t$, the geometric  series $\sum_{\ell_j=\ell_i}^{k-1}
\binom{\ell_i}{t}n^{\ell_j-t-1}p^{\ell_j-t}=O(n^{k-t-2}p^{k-t-1})$ because
the sum of this expression over $\ell_j\geqslant \ell_i$ is bounded by
an increasing geometric series with common ratio $d=np$ dominated by the term when
$\ell_j=k-1$. Secondly,  the geometric  series $\sum_{t=1}^{\ell_i -1}
n^{k-t-2}p^{k-t-1}=O(n^{k-3}p^{k-2})$ because
the sum of this expression over $t\geqslant 1$ is bounded by
a decreasing geometric series with common ratio $1/d$ dominated by the term when
$t=1$. Thus, for any fixed integer $k\geqslant 3$ and $d\geqslant c$
for some large enough constant $c$, we have % and $p=o(1)$,
\begin{align*}
\sum_{\ell_j=\ell_i}^{k-1}
\sum_{t=1}^{\ell_i-1}\binom{\ell_i}{t}n^{\ell_j-t-1}p^{\ell_j-t}=O\bigl(n^{k-3}p^{k-2}\bigr).
\end{align*}
At last, the sum of the expression $\sum_{\ell_i=2}^{k-1}n^{\ell_i-1}p^{\ell_i}$ over $\ell_i\geqslant 2$ is bounded by
an increasing geometric series with common ratio $d=np$ dominated by the term when
$\ell_i=k-1$, thus $\sum_{\ell_i=2}^{k-1}n^{\ell_i-1}p^{\ell_i}=O(n^{k-2}p^{k-1})$.
By the equation in (2.5),
we  have
\begin{align}
\Delta=O\bigl(n^{2k-5}p^{2k-3}\bigr).
\end{align}
Since $n^{2k-5}p^{2k-3}=o(n^{k-3}p^{k-2})$ when $k\geqslant 3$ and $d^{k-1}=o(n)$,
combining the equations in (2.4), (2.6) and Lemma 2.2,  we have
\begin{align*}
\mathbb{P}[d_{G}(u,v)\geqslant
k]=(1-p_{d})\exp\bigl[{O\bigl(n^{k-3}p^{k-2}+n^{2k-4}p^{2k-2}\bigr)}\bigr].
\end{align*}

(b)\ Firstly, there exists a penalty factor $p^{m}$ in the final
count because the edges of ${M}$ are present in $G\in \mathcal{G}(n,p)$.
For $k=2$,  $\mathbb{P}[{M} \ {\rm is\ a}\  k{\mbox-}{\rm matching\ in}\
G]=p^m(1-p)^{4 \binom{m}{2}}$.
For $k\geqslant 3$, we consider $\Delta^\prime$ and $U^\prime$
 in Lemma 2.3 because $ \mathbb{P}[{M} \ {\rm is\ a}\  k{\mbox-}{\rm matching\ in}\
G]=p^m\mathbb{P}[\cap_{i=1}^{l_2} \overline{\mathcal{B}}_i]$.
Following the same discussions between the equation in (2.2) and the equation in (2.4), we have
\begin{align}
U^\prime&=\Bigl[(1-p)\prod\limits_{i=2}^{k-1}(1-p^i)^{(n-4)\cdots(n-i-2)}\Bigr]^{4 \binom{m}{2}}\notag\\
&=\bigl(1-\sum\limits_{i=1}^{k-1} n^{i-1}p^i\bigr)^{4 \binom{m}{2}}\exp\bigl[O(m^2n^{2k-4}p^{2k-2})\bigr]\notag\\
&=\bigl(1-p_{d}\bigr)^{4 \binom{m}{2}}\exp\bigl[O\bigl(m^2n^{k-3}p^{k-2}+m^2n^{2k-4}p^{2k-2}\bigr)\bigr].
\end{align}

Now we  bound $\Delta^\prime=\sum_{i\sim j}\mathbb{P}[\mathcal{B}_i\cap \mathcal{B}_j]$  in Lemma 2.3 below. Similarly,
assume that one path between two  edges in $M$, denoted by $P_i'$, is of length $\ell_i$ satisfying
 $2\leqslant\ell_i\leqslant k-1$. Let $P_j'$ be another path of length $\ell_j$ between two edges in $M$
with $2\leqslant\ell_j\leqslant k-1$.
Let the size of common edges of the paths $P_i'$ and $P_j'$ be denoted
by $t$ with $t\geqslant 1$. The ways to choose $P_i'$ in $K_n$ are at most $4 \binom{m}{2}n^{\ell_i-1}$
and $\mathbb{P}[\mathcal{B}_i]=p^{\ell_i}$.
Fix the path $P_i'$, for a given set of
 $t$ edges in $P_i'$,
then the ways to choose $P_j'$ in $K_n$ are at most $4 \binom{m}{2}n^{\ell_j-t-2}$
if both end-vertices of the path  $P_j'$
 do not belong to the edges where the ones of the path $P_i'$ are, or at most
 $2mn^{\ell_j-t-1}$ if  one end-vertex of the path $P_j'$ belongs to the edges
 where one of the end-vertices of the path $P_i'$ is. By
$\mathbb{P}[\mathcal{B}_j| \mathcal{B}_i]=p^{\ell_j-t}$
 and $4 \binom{m}{2} n^{\ell_j-t-2}p^{\ell_j-t}<2mn^{l_j-t-1}p^{l_j-t}$ when
 $m=O( \frac{n\log d}{d^{k-1}})$,  we have
\begin{align}
\Delta^\prime=\sum_{i\sim j}\mathbb{P}[\mathcal{B}_j|\mathcal{B}_i]
\cdot\mathbb{P}[\mathcal{B}_i]&\leqslant \sum\limits_{\ell_i=2}^{k-1}4 \binom{m}{2}
n^{\ell_i-1}p^{\ell_i}\sum\limits_{\ell_j=\ell_i}^{k-1}
\sum\limits_{t=1}^{\ell_i-1} \binom{\ell_i}{t}2mn^{\ell_j-t-1}p^{\ell_j-t}. %\\
\end{align}
Likewise, following the
same discussions between the equation in (2.5) and the equation in (2.6),
for any fixed integer $k\geqslant 3$, $d\geqslant c$ for some large enough constant $c$
and $p=o(1)$, we have
\begin{align*}
\sum_{\ell_j=\ell_i}^{k-1}
\sum_{t=1}^{\ell_i-1}\binom{\ell_i}{t}2mn^{\ell_j-t-1}p^{\ell_j-t}=O\bigl(mn^{k-3}p^{k-2}\bigr)
\end{align*}
because the main term in this series is when
$\ell_j=k-1$ and $t=1$. By the equation in (2.8), we further have
\begin{align}
\Delta^\prime=O\bigl(m^3n^{2k-5}p^{2k-3}\bigr)
\end{align}
because the main term in $\sum_{\ell_i=2}^{k-1}4 \binom{m}{2}
n^{\ell_i-1}p^{\ell_i}$ is $4 \binom{m}{2}
n^{k-2}p^{k-1}$ when $\ell_i=k-1$.
By the equations in (2.7), (2.9) and Lemma 2.3, the proof of (b) is complete.
\end{proof}

\begin{lemma}Let $k\geqslant2$ a fixed integer, $d=np\geqslant c$ and $d^{k-1}=o(n)$
for some large enough constant $c$. Define $p_d$ to satisfy
 the equation $np_{d}=d^{k-1}$ and %
\begin{equation*}
f(x)= \frac{1}{\sqrt{2\pi x}}\Bigl( \frac{ed n}{2x}\Bigr)^{x}(1-p_{d})^{2x(x-1)}.
\end{equation*}
Then $f(x)$ is increasing in $x$ when $1\leqslant x\leqslant  \frac{(k-1)n\log d}{4d^{k-1}}$.
\end{lemma}

\begin{proof}\quad
Given $n$ and $d$ satisfying the assumption, take the logarithm to $f(x)$.
Differentiate $\log(f(x))$ on $x$. Thus, we have
\begin{align}
f'(x)= f(x)\Bigl[- \frac{1}{2x}+\log\Bigl( \frac{ed n}{2x}\Bigr)-1+(4x-2)\log(1-p_{d})\Bigr].
\end{align}
Let
\begin{align}
g(x)=- \frac{1}{2x}+\log\Bigl( \frac{ed n}{2x}\Bigr)-1+(4x-2)\log(1-p_{d}).
\end{align}
Note that
\begin{align*}
g^\prime(x)&= \frac{1}{2x^2}- \frac{1}{x}+ 4\log(1-p_{d})<0
\end{align*}
because $x\geqslant  1$ and $\log(1-p_{d})\sim -p_d=- d^{k-1}/n$ when $p_d=o(1)$,
then we have
$g(x)$ is a decreasing function in $x$. The function $g(x)$ takes its
minimum when $x= \frac{(k-1)n\log d}{4d^{k-1}}$, that is
 \begin{align}
 g(x)&\geqslant g\Bigl( \frac{(k-1)n\log d}{4d^{k-1}}\Bigr)\notag\\
 &=- \frac{2d^{k-1}}{(k-1)n \log d}+\log \Bigl( \frac{2e d^k}{(k-1)\log d}\Bigr)
 -1+\Bigl( \frac{(k-1)n\log d}{d^{k-1}}-2\Bigr)\log(1-p_{d})\notag\\
 &=- \frac{2d^{k-1}}{(k-1)n \log d}+\log  \frac{2}{(k-1)}+k\log d-\log\log d\notag
 \\&\quad\quad+\Bigl( \frac{(k-1)n\log d}{d^{k-1}}-2\Bigr)\log(1-p_{d}),
 \end{align}
where $\frac{d^{k-1}}{(k-1)n \log d}=o(1)$ by $d^{k-1}=o(n)$, and
\begin{align}
&\Bigl( \frac{(k-1)n\log d}{d^{k-1}}-2\Bigr)\log(1-p_{d})\notag\\
&\sim - \frac{(k-1)n\log d}{d^{k-1}}p_{d}\notag\\
&=- {(k-1)\log d}
\end{align}by  $\log(1-p_{d})\sim -p_d$ and $np_d=d^{k-1}$.
Then we further have
\begin{align*}
 g(x)&\geqslant  \log \frac{2}{k-1}+\log d-\log\log d+ o(1)>0
 \end{align*}
 because $d\geqslant c$ for some large enough constant $c$.
It follows that $f'(x)>0$ from the equations in (2.10) and (2.11).

 The proof of Lemma~2.5 is complete.
\end{proof}

\begin{remark}It is necessary to mention that we can not show the monotonicity of $f(x)$ when
$x\leqslant  \frac{kn\log d}{4d^{k-1}}$ based on the proof of Lemma~2.5
because $g(x)\sim \log \frac{2}{k}-\log\log d<0$ when $x=  \frac{kn\log d}{4d^{k-1}}$
and $d\geqslant c$ for some large enough constant $c$.
\end{remark}

%At last, we will use Talagrand's inequality in the form which appears in~\cite{alon08,oliver21}.
%It is a useful tool to show that a random
%variable is tightly concentrated under certain conditions.

%\begin{definition}
%Let $\Omega=\Pi_{i=1}^n\Omega_i$ be a product of probability spaces such that $\Omega$ has the product
%measure. Let $\sigma:\Omega\rightarrow \mathbb{R}$ and $\theta:\mathbb{N}\rightarrow \mathbb{N}$
%be functions.

%\noindent $\bullet$\ We say that $\sigma$ is Lipschitz if $|\sigma(x)-\sigma(y)|\leqslant 1$
%for every $x,y\in\Omega$ which differ in at most one coordinate.

%\noindent $\bullet$\ We say that $\sigma$ is $\theta$-certifiable if for any $x\in\Omega$ and
%$\xi\in\mathbb{N}$ such that $\sigma(x)\geqslant \xi$, these exists a set of coordinates
%$I\subset [n]$ with $|I|\leqslant \theta(\xi)$ such that each $y\in\Omega$ which agrees with
%$x$ on $I$ also satisfies $\sigma(y)\geqslant \xi$.
%\end{definition}

%\begin{lemma}[Talagrand's inequality in~\cite{alon08,oliver21}] Let $X$ be a Lipschitz
%random variable on $\Omega$ which is $\theta$-certifiable.
%Then for all $\lambda>0$ and $b\in \mathbb{N}$ it holds that
%\begin{align*}
%\mathbb{P}\bigl[X< b-\lambda\sqrt{\theta(b)}\bigr]\cdot\mathbb{P}
%\bigl[X\geqslant b\bigr]\leqslant\exp\Bigl[- \frac{\lambda^2}{4}\Bigr].
%\end{align*}
%\end{lemma}

\vskip20pt
\section{Upper bound  in Theorem~1.1} %

Kang and Manggala~\cite{kang12}
showed an upper bound of $um_k(G)$ when the expected degree $d=np$ is a  large enough constant $c$,
which improves the result in~\cite{andr04}. %a large constant
In fact, we generalize the range of the expected degree $d=np$ based on their proof from $d=c$ to
$d\geqslant c$ and  $d^{k-1}=o(n)$ for some large enough constant $c$ with better analysis.

Let $\mathcal{M}$ be the set of matchings with size $m$ in  $K_n$,  $M_i\in \mathcal{M}$
for $1\leqslant i\leqslant t$ be all matchings in $\mathcal{M}$, where
\begin{align*}
t={n\choose 2m}{2m\choose
{2,\cdots,2}}\frac{1}{m!}.
\end{align*}
Let $I_i$ be the indicator
random variable of the event that
${M}_i$ is a $k$-matching in $G$. Clearly, $X_m=\sum_{i=1}^tI_i$. Thus, we  have
\begin{align}
\mathbb{E}[X_m]&= \binom{n}{2m} \binom{2m}{{2,\cdots,2}} \frac{1}{m!}\mathbb{P}\bigl[{M_i}\ {\rm is\ a\
} k{\mbox-}{\rm
matching}\bigr].
\end{align} %We prove the upper bound in Theorem~1.1 below. %of $um_k(G)$

\begin{proof}[Proof of Upper Bound in Theorem 1.1]\ Let
\begin{align}
m= \bigl(1+o(1)\bigr)\frac{k n \log d}{2d^{k-1}}.
\end{align}
It is easy to verify that $m^2n^{k-3}p^{k-2}=o(m^3n^{2k-5}p^{2k-3})$
when $d\geqslant c$ for some large enough constant $c$.

By the equation in (3.1) and Lemma 2.4 (b),  we have
\begin{equation*}
\begin{aligned}[b]
\mathbb{E}[X_m]
&= \binom{n}{2m} \binom{2m}{{2,\cdots,2}} \frac{1}{m!}p^{m}(1-p_{d})^{4 \binom{m}{2}}
\exp\bigl[O(m^2n^{2k-4}p^{2k-2}+m^3n^{2k-5}p^{2k-3})\bigr]\\
&= \frac{n!}{m!(n-2m)!}\Bigl( \frac{p}{2}\Bigr)^m(1-p_{d})^{2m(m-1)}
\exp\bigl[O(m^2n^{2k-4}p^{2k-2}+m^3n^{2k-5}p^{2k-3})\bigr].
\end{aligned}
\end{equation*}
Note that $\frac{n!}{(n-2m)!}\sim n^{2m}$ when $m$ is in (3.2), $m!\sim \sqrt{2\pi m}(m/e)^m
$ when $m\rightarrow\infty$ by the Stirling formula,
 we further have
\begin{equation*}
\begin{aligned}[b]
\mathbb{E}[X_m]
&\sim \frac{n^{2m}}{\sqrt{2\pi m}}\Bigl(\frac{ep}{2m}\Bigr)^m(1-p_{d})^{2m(m-1)}
\exp\bigl[O(m^2n^{2k-4}p^{2k-2}+m^3n^{2k-5}p^{2k-3})\bigr]\\
&= \frac{1}{\sqrt{2\pi m}}\Bigl( \frac{epn^2}{2m}\Bigr)^m(1-p_{d})^{2m(m-1)}
\exp\bigl[O(m^2n^{2k-4}p^{2k-2}+m^3n^{2k-5}p^{2k-3})\bigr]\\
&= \frac{1}{\sqrt{2\pi m}}\Bigl( \frac{ed n}{2m}\Bigr)^m(1-p_{d})^{2m(m-1)}
\exp\bigl[O(m^2n^{2k-4}p^{2k-2}+m^3n^{2k-5}p^{2k-3})\bigr],
\end{aligned}
\end{equation*}
where the last equality is true because  $d=np$.
Note that $1-p_d\sim \exp[-p_d]$ for  $p_d=o(1)$,
\begin{equation*}
\begin{aligned}[b]
\mathbb{E}[X_m]
&\sim \frac{1}{\sqrt{2\pi
m}}\exp\Bigl[m\Bigl(\log\Bigl(\frac{ed n}{2m}\Bigr)-2(m-1)p_{d}
+O\bigl(mn^{2k-4}p^{2k-2}+m^2n^{2k-5}p^{2k-3}\bigr)\Bigr)\Bigr].
\end{aligned}
\end{equation*}
By the equation in (3.2), we have
\begin{align*}
O\bigl(mn^{2k-4}p^{2k-2}\bigr)=O\Bigl( \frac{d^{k-1}\log
d}{n}\Bigr),\quad
 O\bigl(m^2n^{2k-5}p^{2k-3}\bigr)=O\Bigl( \frac{\log^2
d}{d}\Bigr).
\end{align*}
Note that $d^{k-1}\log d=o(n)$ because
\begin{align*}
&(k-1)\log d +\log\log d\notag\\&= (k-1)\log d\cdot
\Bigl[1+ O\Bigl( \frac{\log\log d}{\log d}\Bigr) \Bigr]\notag\\&=o(\log n)
\end{align*}
is true when $d\geqslant c$ and $d^{k-1}=o(n)$ for some large enough constant $c$, then we have
$O(mn^{2k-4}p^{2k-2})=o(1)$ and $O(m^2n^{2k-5}p^{2k-3})=o(1)$.
It follows that
\begin{align}
\mathbb{E}[X_m]
&\sim \frac{1}{\sqrt{2\pi
m}}\exp\Bigl[m\Bigl(\log\Bigl(\frac{ed n}{2m}\Bigr)-2(m-1)p_{d}+o(1)\Bigr)\Bigr].
\end{align}

%because $m\rightarrow\infty$ in (3.2) when $n\rightarrow\infty$
Likewise, by the equation in  (3.2), we also have
\begin{equation*}
\begin{aligned}[b]
\log\Bigl( \frac{ed n}{2m}\Bigr)&\sim k\log d-\log\log d,\\
 2(m-1)p_{d}&\sim k\log d;
\end{aligned}
\end{equation*}
and then
\begin{align*}
\log\Bigl( \frac{ed n}{2m}\Bigr)-2(m-1)p_{d}+o(1)\sim-\log\log d<0
\end{align*}
in  (3.3) such that $\mathbb{E}[X_m]\rightarrow 0$.
Thus, %by Markov's inequality
$\mathbb{P}[X_m>0]\rightarrow 0$, which implies
w.h.p., for any $k$-matching $M_k$,
\begin{equation*}
e(M_k)\leqslant \bigl(1+o(1)\bigr) \frac{k n \log d}{2d^{k-1}}.
\end{equation*}
The proof of the upper bound in Theorem 1.1 is complete.
\end{proof}

\section{Lower bound in Theorem~1.1}

In this section, we show a lower bound of %$um_k(G)$ by discussing
the size of any maximal $k$-matching in Theorem~1.1.
Let $\mathcal{M}$ be the set of matchings with size $m$ in  $K_n$,  $M\in \mathcal{M}$
and $G\in \mathcal{G}(n,p)$.
Let $Y_m$ denote the
number of maximal $k$-matchings with size $m$ contained in $G$. Thus, we
have
\begin{align}
\mathbb{E}[Y_m]&=\binom{n}{2m} \binom{2m}{{2,\cdots,2}} \frac{1}{m!}\mathbb{P}\bigl[M\ {\rm is\ a\ maximal\
} k{\mbox-}{\rm
matching\ in\ }G\bigr].
\end{align}
For any given real number $\epsilon>0$, define
\begin{align}
m^*= (k-1-\epsilon) \frac{n\log d}{4d^{k-1}}.
\end{align}
In order to show the lower bound $e(M_k)\geqslant (1+o(1))\frac{(k-1)n \log d}{4d^{k-1}}$
for any maximal $k$-matching $M_k$ in Theorem~1.1 is true,
we only need to show $\mathbb{E}[\sum_{m\leqslant m^*}Y_m]\rightarrow 0$
when $n\rightarrow \infty$ because
\begin{align*}
{\mathbb P}\Bigl[\sum_{m\leqslant m^*}Y_m>0\Bigr]\leqslant
\mathbb{E}\Bigl[\sum_{m\leqslant m^*}Y_m\Bigr]
\end{align*}
by Markov's inequality. Then, for any maximal $k$-matching $M_k$ in $G$,
w.h.p., we have $e(M_k)> m^*$.

\begin{proof}[Proof of Lower bound in Theorem~1.1]\ Fix $m$ to be any positive integer satisfying
$m\leqslant m^*$. Let $M$ be a matching with size $m$ in $K_n$.
In fact, the theorem follows from the claim below.

\vskip 0.3cm

\noindent{\bf Claim 4.1 }\quad With the assumption in Theorem~1.1, define $p_d$ to satisfy the equation
$np_{d}=d^{k-1}$, then we have
\begin{equation*}
\begin{aligned}[b]
\mathbb{P}\bigl[M&\ {\rm is\ a\ maximal\
} k{\mbox-}{\rm
matching\ in\ }G\bigr]\\
&\leqslant2p^{m}(1-p_{d})^{4 \binom{m}{2}}\\
&{~~\qquad}\times\exp\Bigl[- \frac{n}{8}d^{-(k-1-\epsilon)}+
O\bigl(m^2n^{k-3}p^{k-2}+m^2n^{2k-4}p^{2k-2}+m^3n^{2k-5}p^{2k-3}\bigr)\Bigr].
\end{aligned}
\end{equation*}
We leave the proof of the above claim later. In fact, if the claim is true,
by the equation in (4.1), then we have
\begin{align}
\mathbb{E}[Y_m]
&<2 \binom{n}{2m} \binom{2m}{{2,\cdots,2}} \frac{1}{m!}p^{m}(1-p_d)^{4{m \choose 2}}\notag\\
&{~~\qquad}\times\exp\Bigl[- \frac{n}{8}d^{-(k-1-\epsilon)}+
O\bigl(m^2n^{k-3}p^{k-2}+m^2n^{2k-4}p^{2k-2}+m^3n^{2k-5}p^{2k-3}\bigr)\Bigr]\notag\\
&= \frac{2n!}{m!(n-2m)!}\Bigl( \frac{p}{2}\Bigr)^m(1-p_d)^{2m(m-1)}\notag\\
&{~~\qquad}\times\exp\Bigl[- \frac{n}{8}d^{-(k-1-\epsilon)}+
O\bigl(m^2n^{k-3}p^{k-2}+m^2n^{2k-4}p^{2k-2}+m^3n^{2k-5}p^{2k-3}\bigr)\Bigr]\notag\\
&< \frac{2n^{2m}}{m!}\Bigl( \frac{p}{2}\Bigr)^m(1-p_d)^{2m(m-1)}\notag\\
&{~~\qquad}\times\exp\Bigl[- \frac{n}{8}d^{-(k-1-\epsilon)}+
O\bigl(m^2n^{k-3}p^{k-2}+m^2n^{2k-4}p^{2k-2}+m^3n^{2k-5}p^{2k-3}\bigr)\Bigr].
\end{align}

By the Stirling formula, $m!\geqslant \sqrt{2\pi m}(m/e)^m
$ for all $m\geqslant 1$, the equation in (4.3) and $d=np$, we further have
\begin{align*}
\mathbb{E}[Y_m]
&< \frac{2n^{2m}}{\sqrt{2\pi m}}\Bigl( \frac{ep}{2m}\Bigr)^m(1-p_d)^{2m(m-1)}\notag\\
&{~~\qquad}\times\exp\Bigl[- \frac{n}{8}d^{-(k-1-\epsilon)}+
O\bigl(m^2n^{k-3}p^{k-2}+m^2n^{2k-4}p^{2k-2}+m^3n^{2k-5}p^{2k-3}\bigr)\Bigr]\notag\\
&= \frac{2}{\sqrt{2\pi m}}\Bigl( \frac{epn^2}{2m}\Bigr)^m(1-p_d)^{2m(m-1)}\notag\\
&{~~\qquad}\times\exp\Bigl[- \frac{n}{8}d^{-(k-1-\epsilon)}+
O\bigl(m^2n^{k-3}p^{k-2}+m^2n^{2k-4}p^{2k-2}+m^3n^{2k-5}p^{2k-3}\bigr)\Bigr]\notag\\
&= \frac{2}{\sqrt{2\pi m}}\Bigl( \frac{ed n}{2m}\Bigr)^m(1-p_d)^{2m(m-1)}\notag\\
&{~~\qquad}\times\exp\Bigl[- \frac{n}{8}d^{-(k-1-\epsilon)}+
O\bigl(m^2n^{k-3}p^{k-2}+m^2n^{2k-4}p^{2k-2}+m^3n^{2k-5}p^{2k-3}\bigr)\Bigr].
\end{align*}

Since $m^*< \frac{(k-1)n\log d}{4d^{k-1}}$  in (4.2),
by Lemma 2.5 and for all $1\leqslant m\leqslant m^*$,  it follows that
\begin{equation*}
 \frac{1}{\sqrt{2\pi m}}\Bigl( \frac{ed n}{2m}\Bigr)^{m}(1-p_d)^{2m(m-1)}
 \leqslant \frac{1}{\sqrt{2\pi m^*}}\Bigl( \frac{ed n}{2m^*}\Bigr)
 ^{m^*}(1-p_d)^{2m^*(m^*-1)},
\end{equation*} and
\begin{align}
\mathbb{E}[Y_m]
&<  \frac{2}{\sqrt{2\pi m^*}}\Bigl( \frac{ed n}{2m^*}\Bigr)^{m^*}(1-p_d)^{2m^*(m^*-1)}\notag\\
&{~~\qquad}\times\exp\Bigl[- \frac{n}{8}d^{-(k-1-\epsilon)}+
O\bigl({m^*}^2n^{k-3}p^{k-2}+{m^*}^2n^{2k-4}p^{2k-2}+{m^*}^3n^{2k-5}p^{2k-3}\bigr)\Bigr]\notag\\
&=  \frac{2}{\sqrt{2\pi m^*}}\Bigl( \frac{ed n}{2m^*}\Bigr)^{m^*}(1-p_d)^{2m^*(m^*-1)}\notag\\
&{~~\qquad}\times\exp\Bigl[- \frac{n}{8}d^{-(k-1-\epsilon)}+
O\bigl({m^*}^2n^{2k-4}p^{2k-2}+{m^*}^3n^{2k-5}p^{2k-3}\bigr)\Bigr],
\end{align}
where the last equality is true because ${m^*}^2n^{k-3}p^{k-2}=o({m^*}^3n^{2k-5}p^{2k-3})$ when
$m^*= (k-1-\epsilon) \frac{n\log d}{4d^{k-1}}$  in (4.2) and $d\geqslant c$
for some large enough constant $c$. For all  $1\leqslant m\leqslant m^*$,
adding the corresponding  inequalities in (4.4) together,
\begin{align*}
\mathbb{E}\biggl[\sum_{m\leqslant m^*}Y_m\biggr]
&< \frac{2m^*}{\sqrt{2\pi m^*}}\Bigl(\frac{ed n}{2m^*}\Bigr)^{m^*}(1-p_d)^{2m^*(m^*-1)}\notag\\
&{~~\qquad}\times\exp\Bigl[- \frac{n}{8}d^{-(k-1-\epsilon)}+
O\bigl({m^*}^2n^{2k-4}p^{2k-2}+{m^*}^3n^{2k-5}p^{2k-3}\bigr)\Bigr]\notag\\
&=\sqrt{ \frac{2m^*}{\pi}} \Bigl( \frac{ed n}{2m^*}\Bigr)
^{m^*}(1-p_d)^{2m^*(m^*-1)}\notag\\&
{~~\qquad}\times
\exp\Bigl[- \frac{m^*d^\varepsilon}{2(k-1-\epsilon)\log
d}+O\bigl({m^*}^2n^{2k-4}p^{2k-2}+{m^*}^3n^{2k-5}p^{2k-3}\bigr)\Bigr],
\end{align*}
where the last equality is true because
\begin{align*}
\exp\Bigl[- \frac{n}{8}d^{-(k-1-\epsilon)}\Bigr]&= \exp\biggl[- \frac{m^*d^\varepsilon}{2(k-1-\epsilon)\log
d}\biggr]
\end{align*}when $m^*= (k-1-\epsilon) \frac{n\log d}{4d^{k-1}}$  in (4.2).
By $1-p_d\sim \exp[-p_d]$ for $p_d=o(1)$, %we further have
\begin{align}
\mathbb{E}\biggl[\sum_{m\leqslant m^*}Y_m\biggr]
&<\sqrt{ \frac{2m^*}{\pi}}\exp\biggl[m^*\Bigl(\log\Bigl( \frac{ed n}{2m^*}\Bigr)-2(m^*-1)p_d\notag\\
&{~~\qquad}- \frac{d^\epsilon}{2(k-1-\epsilon)\log
d}+O\bigl({m^*}n^{2k-4}p^{2k-2}+{m^*}^2n^{2k-5}p^{2k-3}\bigr)\Bigr)\biggr].
\end{align}
Again by $m^*=(k-1-\epsilon) \frac{n\log d}{4d^{k-1}}\rightarrow\infty$ in (4.2),
 $d^{k-1}=np_d$ and $d\geqslant c$  for some large enough constant $c$,
it follows that
\begin{align}
\log\Bigl( \frac{ed n}{2m^*}\Bigr)&=\log\Bigl( \frac{2ed^k}{(k-1-\epsilon)\log d}\Bigr)\sim
k\log d,\\
2(m^*-1)p_d&\sim \frac{1}{2}(k-1-\epsilon)\log d.
\end{align}Similarly, we also have
\begin{align}
{{m^*}^2}n^{2k-5}p^{2k-3}&=
\frac{(k-1-\epsilon)^2(\log d)^2}{16d}=o(1),\\
{m^*}n^{2k-4}p^{2k-2}&=
\frac{(k-1-\epsilon)d^{k-1}\log d}{4n}=o(1).
\end{align}

Putting the equations in (4.6)-(4.9) into the equation in (4.5),
we have
\begin{equation*}
\begin{aligned}[b]
&\log\Bigl( \frac{ed n}{2m^*}\Bigr)-2(m^*-1)p_d- \frac{d^\epsilon}{2(k-1-\epsilon)\log
d}+
O\bigl(m^*n^{2k-4}p^{2k-2}+{m^*}^2n^{2k-5}p^{2k-3}\bigr)<0
\end{aligned}
\end{equation*}
because $  \frac{d^\epsilon}{2(k-1-\epsilon)\log
d}$ dominates all these terms
when $d\geqslant c$  for some large enough constant $c$.
Finally, using the equation in (4.5), we have
\begin{align*}
\mathbb{E}\Bigl[\sum_{m\leqslant m^*}Y_m\Bigr]\rightarrow 0,
\end{align*}
%because $m^*\rightarrow\infty$ when $n\rightarrow\infty$,
and then
\begin{align*}
{\mathbb P}\Bigl[\sum_{m\leqslant m^*}Y_m>0\Bigr]\rightarrow 0
\end{align*}
by Markov's inequality. W.h.p. we have $e(M_k)>m^*$ for any maximal $k$-matching $M_k$ in $G$.

 The proof of the lower bound  in Theorem~1.1 is complete.\end{proof}

In order to formally finish the  proof of  Theorem~1.1,
it is necessary to prove  Claim~4.1.

\begin{proof}[Proof of Claim 4.1]\ Recall that  $m$ is a fixed positive integer satisfying
$m\leqslant m^*$ and $M$ is a matching with size $m$ in $K_n$. Let
$\Gamma_{\geqslant k}({M})$ denote the set of vertices in
$G\in \mathcal{G}(n,p)$  whose distances are at least $k$ to each vertex in ${M}$. If ${M}$ is a maximal $k$-matching
in $G$, then $\Gamma_{\geqslant k}({M})$ is either empty or an independent
set, otherwise it contradicts with the maximal property of ${M}$.
Define two events $\mathcal{E}$ and $\mathcal{F}$ as
\begin{align}
\mathcal{E}=\{\Gamma_{\geqslant k}({M})\ {\rm is\ an\
independent\ set}\}\quad {\rm{and}}\quad \mathcal{F}=\Bigl\{\bigl|\Gamma_{\geqslant
k}({M})\bigr|> \frac{n}{2}d^{-\frac{(k-1-\epsilon)}{2}}\Bigr\}.
\end{align}
By the total probability formula, we have
\begin{align*}
&\mathbb{P}\bigl[M\ {\rm is\ a\ maximal\
} k{\mbox-}{\rm
matching\ in\ }G\bigr]\notag\\
&=\mathbb{P}\bigl[M\ {\rm is\ a\ maximal\
} k{\mbox-}{\rm
matching\ in\ }G|\mathcal{F}\bigr]\cdot P\bigl[\mathcal{F}\bigr]\notag\\
&\qquad+\mathbb{P}\bigl[M\ {\rm is\ a\ maximal\
} k{\mbox-}{\rm
matching\ in\ }G|\mathcal{F}^c\bigr]\cdot P\bigl[\mathcal{F}^c\bigr],
\end{align*}
in which $\mathbb{P}[M\ {\rm is\ a\ maximal\
} k{\mbox-}{\rm
matching\ in\ }G|\mathcal{F}]\cdot P[\mathcal{F}]=\mathbb{P}[\{M\ {\rm is\ a\
} k{\mbox-}{\rm
matching\ in\ }G\}\cap \mathcal{E}\cap \mathcal{F}]$ by the definitions of the events
$\mathcal{E}$, $\mathcal{F}$ and the maximal property of ${M}$; and it is clear that
$\mathbb{P}[\{M\ {\rm is\ a\ maximal\
} k{\mbox-}{\rm
matching\ in\ }G\}\cap\mathcal{F}^c]\leqslant \mathbb{P}[\{M\ {\rm is\ a\
} k{\mbox-}{\rm
matching\ in\ }G\}\cap \mathcal{F}^c]$. It follows that
\begin{align*}
&\mathbb{P}\bigl[M\ {\rm is\ a\ maximal\
} k{\mbox-}{\rm
matching\ in\ }G\bigr]\notag\\
&\leqslant\mathbb{P}\bigl[\{M\ {\rm is\ a\
} k{\mbox-}{\rm
matching\ in\ }G\}\cap \mathcal{E}\cap \mathcal{F}\bigr]+\mathbb{P}\bigl[\{M\ {\rm is\ a\
} k{\mbox-}{\rm
matching\ in\ }G\}\cap \mathcal{F}^c\bigr]\\
&=\mathbb{P}[M\ {\rm is\ a\
} k{\mbox-}{\rm
matching\ in\ }G]\bigl(\mathbb{P}[\mathcal{E}\cap \mathcal{F}]+\mathbb{P}[\mathcal{F}^c]\bigr),
\end{align*}
where the last equality is true because the event $\{M\ {\rm is\ a\
} k{\mbox-}{\rm matching\ in\ }G\}$ is independent to the events $\mathcal{E}\cap \mathcal{F}$ and
$\mathcal{F}^c$.
Combining with Lemma 2.4 (b), we further have
\begin{align}
&\mathbb{P}[M\ {\rm is\ a\ maximal\
} k{\mbox-}{\rm
matching\ in\ }G]\notag\\
&\leqslant p^{m}(1-p_d)^{4\binom{m}{2}}\exp\bigl[{O(m^2n^{k-3}p^{k-2}+m^2n^{2k-4}p^{2k-2}+m^3n^{2k-5}p^{2k-3})}\bigr]\notag\\
&\qquad\times\bigl(\mathbb{P}[\mathcal{E}\cap \mathcal{F}]+\mathbb{P}[\mathcal{F}^c]\bigr).
\end{align}

Hence, in order to finish the proof of  Claim 4.1, it is necessary to show
\begin{align}
\mathbb{P}[\mathcal{F}^c]&<
\exp\Bigl[- \frac{n}{8}d^{-(k-1-\epsilon)}\Bigr],\\
\mathbb{P}[\mathcal{E}\cap \mathcal{F}]&<\exp\Bigl[- \frac{n}{8}d^{-(k-1-\epsilon)}\Bigr].
\end{align}

Firstly, we prove the inequality in (4.12).
 Fix a vertex $u$  not in $M$. By Lemma 2.4 (a), for any vertex $v$ in $M$,
 $\mathbb{P}[d_{G}(u,v)\geqslant
k]=1-p$ for $k=2$ and $\mathbb{P}[d_{G}(u,v)\geqslant
k]=(1-p_d)\exp[O(n^{k-3}p^{k-2}+n^{2k-4}p^{2k-2})]$ for $k\geqslant 3$. While
 for two different vertices $v_1$ and $v_2$ in $M$,
the events $\{d_{G}(u,v_1)\geqslant k\}$ and $\{d_{G}(u,v_2)\geqslant k\}$
are not independent.
The property of distance is
a decreasing property, which means that  if the event that $\{d_{G}(u,v_1)\geqslant k\}$ is true,
then the probability of the event $\{d_{G}(u,v_2)\geqslant k\}$
is at least the corresponding value in the case of independence.
Since there are exactly $2m$ vertices in $M$, we have the event $\{u\in \Gamma_{\geqslant k}({M})\}$
occurs with probability at least $P$, where
\begin{align}
P&=\begin{cases}(1-p)^{2m},&k=2;\\(1-p_{d})^{2m}\exp[{O(mn^{k-3}p^{k-2}+mn^{2k-4}p^{2k-2})}],&k\geqslant3.
\end{cases}
\end{align}
Hence, the probability of the event that the variable $|\Gamma_{\geqslant
k}({M})|$ is no greater than some value is dominated by
the event that the variable ${\rm{Bin}}(n-2m,P)$ is no greater than the same value, which means
\begin{align}
\mathbb{P}[\mathcal{F}^c]
&=\mathbb{P}\Bigl[|\Gamma_{\geqslant
k}({M})|\leqslant \frac{n}{2}d^{- \frac{(k-1-\epsilon)}{2}}\Bigr]\notag\\
&\leqslant \mathbb{P}\Bigl[{\rm{Bin}}(n-2m,P)\leqslant \frac{n}{2}d^{- \frac{(k-1-\epsilon)}{2}}\Bigr].
\end{align}

For any $1\leqslant m\leqslant m^*$ in (4.2),  $d\geqslant c$ for some large enough constant $c$, % and Remark 2.7,
we have $n-2m\sim n$ and
\begin{align*}
mn^{k-3}p^{k-2}&=O\Bigl( \frac{\log d}{d}\Bigr)=o(1),\\
mn^{2k-4}p^{2k-2}&=O\Bigl( \frac{d^{k-1}\log d}{n}\Bigr)=o(1).
\end{align*}
Thus, we have $\exp[{O(mn^{k-3}p^{k-2}+mn^{2k-4}p^{2k-2})}]\sim 1$  and
$P\sim (1-p_d)^{2m}$ in (4.14). By the equation in (4.15),
\begin{align}
\mathbb{P}[\mathcal{F}^c]&
\leqslant \mathbb{P}\Bigl[{\rm{Bin}}(n-2m,P)\leqslant \frac{n}{2}d^{- \frac{(k-1-\epsilon)}{2}}\Bigr]\notag\\
&\sim \mathbb{P}\Bigl[{\rm{Bin}}(n,(1-p_d)^{2m})\leqslant \frac{n}{2}d^{- \frac{(k-1-\epsilon)}{2}}\Bigr].
\end{align}

The expectation of
${\rm{Bin}}(n,(1-p_d)^{2m})$ is
\begin{align}
\mathbb{E}\bigl[{\rm{Bin}}(n,(1-p_d)^{2m})\bigr]&\geqslant
 n\exp\Bigl[- \frac{2mp_d}{1-p_d}\Bigr]\notag\\
&\geqslant n\exp\Bigl[- \frac{2m^*p_d}{1-p_d}\Bigr]\notag\\
&= n\exp
\biggl[- \frac{ \frac{(k-1-\epsilon)}{2}\log d}{1- \frac{d^{k-1}}{n}}\biggr]\notag\\
&\sim nd^{- \frac{(k-1-\varepsilon)}{2}},
\end{align}
where the first inequality is true because $1-x\geqslant \exp\bigl[{- \frac{x}{1-x}}\bigr]$
for any $0<x<1$; the second inequality is true because $m\leqslant m^*$  in (4.2) and $p_d= d^{k-1}/n$.

Putting $\delta= \frac{1}{2}$ in Lemma 2.1, by the equations in (4.16) and (4.17), we have
\begin{align}
\mathbb{P}[\mathcal{F}^c]&<
\exp\Bigl[- \frac{n}{8} d^{- \frac{(k-1-\epsilon)}{2}}\Bigr]\notag\\
&<\exp\Bigl[- \frac{n}{8} d^{-(k-1-\epsilon)}\Bigr],
\end{align}
where the last inequality is obviously true when $k\geqslant 2$ and any real number $\epsilon>0$.
The proof of the equation in (4.12) is complete.

Secondly, we will prove the equation in (4.13)
to complete the proof of Claim 4.1.
Let $\mathcal{S}$ be the collection of
vertex subsets $S$ in $[n]$ satisfying
$|S|> \frac{n}{2} d^{- \frac{(k-1-\epsilon)}{2}}$,  which implies
\begin{equation*}
\mathcal{S}=\Bigl\{S\subset [n]\ \big|\ |S|> \frac{n}{2} d^{- \frac{(k-1-\epsilon)}{2}}\Bigr \}.
\end{equation*}
Note that
\begin{equation*}
\begin{aligned}[b]
\mathbb{P}[S\ {\rm is\ independent}]&=(1-p)^{ \frac{|S|(|S|-1)}{2}}.
\end{aligned}
\end{equation*}
Since $p= d/n=o(1)$, $1-p\sim \exp[-p]$ and $|S|> \frac{n}{2} d^{- \frac{(k-1-\epsilon)}{2}}
\rightarrow\infty$ when $d^{k-1}=o(n)$,
we  have
\begin{align}
\mathbb{P}[S\ {\rm is\ independent}]
&\sim\exp\Bigl[- \frac{|S|^2-|S|}{2}p\Bigr]\notag\\
&<\exp\Bigl[- \frac{n}{8} d^{-(k-1-\varepsilon)}\Bigr],
\end{align}
where the last inequality is true because
\begin{align*}
\frac{|S|^2-|S|}{2}p\sim\frac{|S|^2}{2}\cdot \frac{d}{n}> \frac{n}{8} d^{-(k-2-\varepsilon)}>
\frac{n}{8} d^{-(k-1-\varepsilon)}
\end{align*} when $|S|\rightarrow\infty$.
At last, by the equation in (4.19), we have
\begin{align*}
\mathbb{P}[\mathcal{E}\cap \mathcal{F}]=\mathbb{P}[\mathcal{E}|\mathcal{F}]
\mathbb{P}[\mathcal{F}]=\sum_{S\in \mathcal{S}}
\mathbb{P}\bigl[S\ {\rm is\ independent}\bigr]\cdot \mathbb{P}[\Gamma_{\geqslant
k}({M})=S],
\end{align*} then
\begin{equation*}
\begin{aligned}[b]
\mathbb{P}[\mathcal{E}\cap \mathcal{F}]&<\exp\Bigl[- \frac{n}{8} d^{-(k-1-\varepsilon)}\Bigr]\cdot\sum\limits_{S\in \mathcal{S}}
\mathbb{P}\bigl[\Gamma_{\geqslant k}({M})=S\bigr]\\
&<\exp\Bigl[- \frac{n}{8} d^{-(k-1-\varepsilon)}\Bigr].
\end{aligned}
\end{equation*}
The proof of the equation in (4.13) is complete.
\end{proof}

\begin{remark}
Why do we define the event $\mathcal{F}$ to be the equation in (4.10)?
In fact, it is not necessary to find the optimal coefficient.
The form of $\mathcal{F}$  mainly helps us obtain the same term $\exp[- \frac{n}{8}d^{-(k-1-\varepsilon)}]$
in (4.18) and (4.19) to finally prove the equations in (4.12) and (4.13).
\end{remark}

\section{Generation of one large $k$-matching}

In this section, by investigating the property for
some given number of vertices  in $G\in \mathcal{G}(n,p)$,
we design a randomized greedy algorithm to generate a $k$-matching
with  size  $(1+o(1))\frac{k n\log d}{4d^{k-1}}$ to finish the proof of Theorem 1.3.

Consider ``The large $k$-matching Generator" randomized greedy algorithm below.

\noindent{\textbf {Algorithm}}:\ The large $k$-matching Generator\\
{\bf Input}: A graph $G$ on the vertex $[n]$ with $n\rightarrow\infty$, $k\geqslant 2$,
$d\geqslant c$ and $d^{k-1}=o(n)$ for some large enough constant $c$.\\
{\bf Output}: A large $k$-matching of $G$.\\
{\bf Step 1}: Uniformly at random choose $2s$ vertices $S$ from $[n]$.
Pair them into $s$ pairs randomly and independently,
where $s= \frac{n}{4d^{k-1}}[k\log d-3\log (k\log d)]$. \\
{\bf Step 2}: If $S$ is not a $k$-matching in $G$, replace one pair by an edge in
$V(G)\setminus S$ at distance at least $k$ to $S$.\\
{\bf Step 3}: Output $S$ if $S$ is a $k$-matching in $G$. Otherwise, go to step 2.

We consider the efficiency of the above algorithm by investigating the property
for any given $2s$ vertices in $G\in \mathcal{G}(n,p)$, which
is guaranteed by the following Theorem 5.1.

\begin{theorem}
Let $k\geqslant2$ be a fixed integer,  $G\in \mathcal{G}(n,p)$  with the expected degree
 $d=np$ and $n\rightarrow\infty$, where $d\geqslant c$  and $d^{k-1}=o(n)$ for some large enough constant $c$.
  W.h.p., there are $ \frac{n}{d^{k/2}}(k\log d)^{3/2}$
 vertices  at distance at least $k$ to  every set of $2s$ vertices, where
 \begin{align}
 s= \frac{n}{4d^{k-1}}\bigl[k\log d-3\log (k\log d)\bigr].
 \end{align}
Moreover, these $ \frac{n}{d^{k/2}}(k\log d)^{3/2}$ vertices induce at least one edge in $G$.
\end{theorem}

\begin{proof}[Proof of Theorem 5.1]\ We prove Theorem 5.1 by two claims in the following.
Let $S\subseteq [n]$ be a vertex subset with $|S|=2s$. Let  $\Gamma_{i}(S)$ be the
set of vertices satisfying
\begin{align*}
\Gamma_{i}(S)=\{w\in V| d(w,S)=i\}
\end{align*}
 for $0\leqslant i\leqslant k-2$. Define the event $\mathcal{F}_i$ to be
 $\mathcal{F}_i=\bigl\{|\Gamma_{i}(S)|\sim 2sd^i\bigr\}$ for $0\leqslant i\leqslant k-2$.
Let $\overline{\mathcal{F}_i}$ be the complement of the event $\mathcal{F}_i$,
which means that there exists some $\epsilon_i\in (0,1)$
such that
\begin{align*}
\bigl||\Gamma_{i}(S)|- 2sd^i\bigr|>\epsilon_i sd^i.
\end{align*}

\noindent {\bf Claim 5.1}\quad Fix a vertex subset $S\subseteq [n]$ with $|S|=2s$.
W.h.p. the events $\mathcal{F}_i$ for $0\leqslant i\leqslant k-2$
 are all true.

 \begin{proof}[Proof of Claim 5.1]\ Note that $\mathcal{F}_0$ is true
 because $|\Gamma_0(S)|=2s$. If $k=2$, we are done. Let $k\geqslant 3$ below.
 Assume the events $\mathcal{F}_{0},\cdots,\mathcal{F}_{i-1}$ are all true
 when $1\leqslant i\leqslant k-2$. Thus, $|\Gamma_{i}(S)|$ is a binomially distributed
 random variable with parameters $n-\sum_{j=0}^{i-1}|\Gamma_{j}(S)|$
 and $1-(1-p)^{|\Gamma_{i-1}(S)|}$ because each vertex in $[n]-\cup_{j=0}^{i-1}\Gamma_{j}(S)$
 is selected into $\Gamma_{i}(S)$ if and only if it is connected with at least one vertex in $\Gamma_{i-1}(S)$,
which occurs  with probability $1-(1-p)^{|\Gamma_{i-1}(S)|}$.  We have
\begin{align}
&\mathbb{P}[\overline{\mathcal{F}_{i}}|\mathcal{F}_{0},\cdots,\mathcal{F}_{i-1}]\notag\\
&=\mathbb{P}\Bigl[\Bigl|{\rm Bin}\Bigl(n-\sum_{j=0}^{i-1}|\Gamma_{j}(S)|,
1-(1-p)^{|\Gamma_{i-1}(S)|}\Bigr)-2sd^i\Bigr|> \epsilon_i sd^i\Bigr].
\end{align}

Since the events $\mathcal{F}_{0},\cdots,\mathcal{F}_{i-1}$ are all true by induction,
we have obtained $|\Gamma_{j}(S)|\sim 2sd^{j}$ for $0\leqslant j\leqslant i-1$ and
\begin{align*}
\sum_{j=0}^{i-1}|\Gamma_{j}(S)|&\sim\sum_{j=0}^{i-1}2sd^{j}
\sim 2sd^{i-1}
\end{align*}
where $\sum_{j=0}^{i-1}2sd^{j}\sim 2sd^{i-1}$ is true  because  $k$ is fixed and
 the sum $\sum_{j=0}^{i-1}2sd^{j}$ is bounded by an increasing geometric series
with common ratio $d$ dominated
by the term when $j=i-1$. By $s=O( \frac{n\log d}{d^{k-1}})$ in (5.1), we further have %of $2sd^{j}$ over
\begin{align*}
\sum_{j=0}^{i-1}|\Gamma_{j}(S)|\sim 2sd^{i-1}=o(n),
\end{align*}
and then
\begin{align}
n-\sum_{j=0}^{i-1}|\Gamma_{j}(S)|\sim n.
\end{align}

On the other hand,  since $|\Gamma_{i-1}(S)|\sim2sd^{i-1}$ by induction, and $s=O( \frac{n\log d}{d^{k-1}})$  in
 (5.1), we also have
\begin{align}
|\Gamma_{i-1}(S)|p= O\Bigl(\frac{\log d}{d^{k-i-1}}\Bigr)=o(1)
\end{align}
when $d\geqslant c$ for some large enough constant $c$ and $0\leqslant i\leqslant k-2$.
By Taylor's expansion,
\begin{align*}
(1-p)^{|\Gamma_{i-1}(S)|}&=\sum_{j=0}^{\infty} \binom{|\Gamma_{i-1}(S)|}{j}\bigl(-p\bigr)^j\\
&=1-|\Gamma_{i-1}(S)|p+O\bigl(|\Gamma_{i-1}(S)|^2p^2\bigr),
\end{align*}
and by the equation  in (5.4),
\begin{align}
1-(1-p)^{|\Gamma_{i-1}(S)|}
&=|\Gamma_{i-1}(S)|p\cdot\bigl(1+O(|\Gamma_{i-1}(S)|p)\bigr)\notag\\
&\sim 2sd^{i-1}p.
\end{align}
Combining  the equations in (5.2), (5.3) and (5.5), we have
\begin{align*}
\mathbb{P}\bigl[\overline{\mathcal{F}_{i}}|\mathcal{F}_{0},\cdots,\mathcal{F}_{i-1}\bigr]\sim\mathbb{P}\bigl[|{\rm Bin}(n,
2sd^{i-1}p)-2sd^i|> \epsilon_i sd^i\bigr].
\end{align*}
Note that $\mathbb{E}[{\rm Bin}(n, 2sd^{i-1}p)]=2sd^{i}$ by $d=np$. By  Lemma 2.1 with $\delta=\epsilon_i$, we have
\begin{align}
\mathbb{P}\bigl[\overline{\mathcal{F}_{i}}|\mathcal{F}_{0},\cdots,\mathcal{F}_{i-1}\bigr]
<2\exp\bigl[- \epsilon_i^2sd^i\bigr]\rightarrow 0
\end{align}
because  $s=\frac{n}{4d^{k-1}}[k\log d-3\log(k\log d)]$  in (5.1), $d^{k-1}=o(n)$
and $0\leqslant i\leqslant k-2$. At last, we have
\begin{align*}
\mathbb{P}\bigl[\mathcal{F}_0,\cdots,\mathcal{F}_{k-2}\bigr]&=\mathbb{P}[\mathcal{F}_0]\prod_{i=1}^{k-2}
\mathbb{P}\bigl[\mathcal{F}_{i}|\mathcal{F}_{0},\cdots,\mathcal{F}_{i-1}\bigr]=\prod_{i=1}^{k-2}
\mathbb{P}\bigl[\mathcal{F}_{i}|\mathcal{F}_{0},\cdots,\mathcal{F}_{i-1}\bigr]\rightarrow 1
\end{align*} because $\mathbb{P}[\mathcal{F}_0]=1$ and $k\geqslant 3$ is a fixed integer.
\end{proof}

\vskip 0.3cm

\noindent {\bf Claim 5.2}\quad For any  $S\subseteq [n]$ such that $|S|=2s$,
w.h.p. there are $ \frac{n}{d^{k/2}}(k\log d)^{3/2}$ vertices  at distance at least $k$
to  $S$. Moreover, these $ \frac{n}{d^{k/2}}(k\log d)^{3/2}$ vertices
induce at least one edge in $G$.

 \begin{proof}[Proof of Claim 5.2]\
Firstly, fix a vertex subset $S\subseteq [n]$ with $|S|=2s$. By Claim 5.1,
we have $\mathbb{P}[\mathcal{F}_{0},\cdots,\mathcal{F}_{k-2}]\rightarrow 1$.
Assume that the events $\mathcal{F}_{0},\cdots,\mathcal{F}_{k-2}$ hold below.
Let $A$ be the value  satisfying
\begin{align}
A= \frac{n}{d^{ k/2}}(k\log d)^{3/2}.
\end{align}Note that $A\rightarrow\infty$ when $n\rightarrow\infty$
because  $k\geqslant 2$ is a fixed integer and $d^{k-1}=o(n)$.
Define the event $\mathcal{E}_{ k}=\{|\Gamma_{\geqslant k}(S)|\sim A \}$,
where $\Gamma_{\geqslant k}({S})$ denotes the set of vertices in
$G$  whose distances are at least $k$ to each vertex in ${S}$.
Then, the complement $\overline{\mathcal{E}_k}$ of  $\mathcal{E}_{k}$ means that there exists some $\epsilon\in (0,1)$
such that
\begin{align*}
\bigl||\Gamma_{\geqslant k}(S)|- A\bigr|>\epsilon A.
\end{align*}

Since the events $\mathcal{F}_{0},\cdots,\mathcal{F}_{k-2}$ hold,
$|\Gamma_{\geqslant k}({S})|$ is a binomially distributed random variable
with parameters $n-\sum_{i=0}^{k-2}|\Gamma_{i}(S)|$ and $(1-p)^{|\Gamma_{k-2}(S)|}$
because each vertex in $[n]-\cup_{i=0}^{k-2}\Gamma_{i}(S)$
 is selected into $\Gamma_{\geqslant k}({S})$ independently if and only if it is not
 connected with any vertex in $\Gamma_{k-2}(S)$,
which occurs  with probability $(1-p)^{|\Gamma_{k-2}(S)|}$.
Thus, we have
\begin{align}
&\mathbb{P}[\mathcal{E}^c_{ k}|\mathcal{F}_{0},\cdots,\mathcal{F}_{k-2}]\notag\\
&=\mathbb{P}\Bigl[\Bigl|{\rm Bin}\Bigl(n-\sum_{i=0}^{k-2}|\Gamma_{i}(S)|,
(1-p)^{|\Gamma_{k-2}(S)|}\Bigr)-A\Bigr|> \epsilon A\Bigr].
\end{align}

Under the condition that the events $\mathcal{F}_{0},\cdots,\mathcal{F}_{k-2}$ hold and $k\geqslant 2$
is a fixed integer, we similarly have $\sum_{i=0}^{k-2}|\Gamma_{i}(S)|\sim \sum_{i=0}^{k-2}2sd^{i}
\sim2sd^{k-2}$ because the sum  is dominated by the term when $i=k-2$. By $s=O(\frac{n\log d}{d^{k-1}})$ in (5.1)
and $d\geqslant c$ for some large enough constant $c$,
we also have $\sum_{i=0}^{k-2}|\Gamma_{i}(S)|=O( \frac{n\log d}{d})=o(n)$, and then
\begin{align}
n-\sum_{i=0}^{k-2}|\Gamma_{i}(S)|&\sim n.
\end{align}
By $p=d/n=o(1)$, $1-p\sim \exp [-p]$ and $k\geqslant 2$ a fixed integer,
\begin{align}
(1-p)^{|\Gamma_{k-2}(S)|}&\sim \exp\bigl[- |\Gamma_{k-2}(S)|\cdot p\bigr]\notag\\
&\sim \exp\bigl[- 2sd^{k-2}p\bigr]\notag\\
&=\exp\Bigl[- \frac{1}{2}k\log d+ \frac{3}{2}\log( k\log d)\Bigr]\notag\\
&= \frac{A}{n},
\end{align}
where $|\Gamma_{k-2}(S)|\sim 2sd^{k-2}= \frac{n}{2d}[k\log d-3\log( k\log d)]$ is true
because the event $\mathcal{F}_{k-2}$ holds;
and the last equation is true because $A= \frac{n}{d^{ k/2}}(k\log d)^{3/2}$ in (5.7).
By Lemma 2.1 with $\delta=\epsilon$, the equations in (5.8), (5.9) and (5.10),  we have
\begin{align*}
&\mathbb{P}[\overline{\mathcal{E}_{ k}}|\mathcal{F}_{0},\cdots,\mathcal{F}_{k-2}]\notag\\
&\sim \mathbb{P}\bigl[\bigl|{\rm Bin}\bigl(n, A/{n}\bigr)-A\bigr|> \epsilon A\bigr]\notag\\
&<2\exp\Bigl[- \frac{\epsilon^2A}{2}\Bigr]\notag\\
&=2\exp\Bigl[- \frac{\epsilon^2n}{2d^{k/2}}\bigl(k\log d\bigr)^{3/2}\Bigr]\\
&\rightarrow 0.
\end{align*}
It follows that $\mathbb{P}[\mathcal{E}_{ k}|\mathcal{F}_{0},\cdots,\mathcal{F}_{k-2}]\rightarrow 1$.
By Claim 5.1, for a fixed vertex subset $S\subseteq [n]$
with $|S|=2s$,
\begin{align}
\mathbb{P}\bigl[\mathcal{F}_{0},\cdots,\mathcal{F}_{k-2},\mathcal{E}_{ k}\bigr]=\mathbb{P}\bigl[\mathcal{E}_{ k}|\mathcal{F}_{0},\cdots,\mathcal{F}_{k-2}\bigr]\cdot\mathbb{P}\bigl[\mathcal{F}_{0},\cdots,\mathcal{F}_{k-2}\bigr]\rightarrow 1.
\end{align}

Furthermore, under the condition that the events $\mathcal{F}_{0},\cdots,\mathcal{F}_{k-2},\mathcal{E}_{ k}$
hold, define the event $\mathcal{K}$ to be
\begin{align*}%
\mathcal{K}=\bigl\{\text{The vertex subset $\Gamma_{\geqslant k}(S)$
induces at least one edge in }G\bigr\}.
\end{align*}
%Thus, the complement $\overline{\mathcal{K}}$ of the event $\mathcal{K}$
%holds with  probability $(1- p\bigr)^{\binom{A}{2}}$.
Since $|\Gamma_{\geqslant k}(S)|\sim A$, $p=o(1)$ and $1-p\sim \exp[-p]$, we further have
\begin{align}
&\mathbb{P}[\overline{\mathcal{K}}|\mathcal{F}_{0},\cdots,\mathcal{F}_{k-2}, \mathcal{E}_{ k}]\notag\\
&\sim(1- p)^{\binom{A}{2}}\notag\\
&\sim \exp\Bigl[- \frac{A^2}{2}p \Bigr]\notag\\
&=\exp\Bigl[- \frac{n}{2d^{k-1}}\bigl(k\log d\bigr)^3 \Bigr],
\end{align}
where the last equality is true by the equation $A= \frac{n}{d^{ k/2}}(k\log d)^{3/2}$ in (5.7).

 %By the equation in (5.14),

On the other hand, by the equation in (5.1), the number of the vertex subsets
$S\subseteq [n]$ with $|S|=2s$ is
\begin{align}
\binom{n}{2s}\leqslant \Bigl( \frac{en}{2s}\Bigr)^{2s}=
\exp\Bigl[2s\log\Bigl( \frac{en}{2s}\Bigr)\Bigr]<\exp\Bigl[ \frac{n}{2d^{k-1}}\bigl(k\log d\bigr)^2\Bigr].
\end{align}
By the equations in (5.12) and (5.13), the union bound of the probability of bad events
is
\begin{align}
 &\binom{n}{2s}\mathbb{P}[\mathcal{K}^c|\mathcal{F}_{0},\cdots,\mathcal{F}_{k-2}, \mathcal{E}_{ k}]\notag\\
 &\leqslant\exp\Bigl[ \frac{n}{2d^{k-1}}\bigl(k\log d\bigr)^2- \frac{n}{2d^{k-1}}\bigl(k\log d\bigr)^3 \Bigr]\notag\\
 &\rightarrow 0
\end{align}
when $d\geqslant c$ for some large enough constant $c$.

Finally, by the equations in (5.11) and (5.14), for any vertex subset $S\subseteq [n]$ of size $|S|=2s$,
w.h.p. there are $ \frac{n}{d^{k/2}}(k\log d)^{3/2}$ vertices  at distance at least $k$
to  $S$, and these $ \frac{n}{d^{k/2}}(k\log d)^{3/2}$ vertices induce at least one edge in $G$.

 We complete the proof of Claim 5.2.
\end{proof}
According to Claim 5.1 and Claim 5.2, we complete the proof of Theorem 5.1.
\end{proof}

\begin{remark}Using the exactly same argument in Theorem 5.1 and taking
$s= \frac{n}{2d^{k-1}}[(k-1)\log d-3\log((k-1)\log d)]$,
it also can be shown that every vertex subset $S$ with $2s$ vertices  has $ \frac{n}{2d^{k-1}}[(k-1)\log d]^3$
vertices at distance at least $k$ to $S$, instead it is not enough to  show these
vertices induce at least one edge. We can't improve Theorem~1.3 through this approach.
%It is not easy to generate a $k$-matching better than Theorem 1.2.
\end{remark}

\begin{remark}
Recently, Cooley et al. in~\cite{oliver21} showed that $um_2(G)=(1+o(1)) \frac{n\log d}{d}$
by the second moment method and Talagrand's inequality
  when  $ d\geqslant{c}$ and $d=o(n)$ for some large enough constant $c$. For $um_k(G)$ when $k\geqslant 3$,
  can we apply the similar approach to obtain better results than the one in Theorem 1.1?
  Unfortunately, it only holds on $um_2(G)$.

 %Unfortunately, it does not work.

Regard $\mathcal{G}(n,p)=\prod_{i=1}^{n-1}Z_i$ as %and $um_k:\mathcal{G}(n,p)\rightarrow \mathbb{R} $
a product of $n-1$ probability spaces $Z_i$ for $i\in[n-1]$, where each $Z_i$
picks uniformly at random a subset of $[i]$ of size ${\rm Bin}(i,p)$ as the neighbours
of the vertex $i+1$ within $[i]$.
In order to apply Talagrand's inequality on the random variable $um_k(G)$,
it should be Lipschitz and $\theta$-certifiable for some function
$\theta:\mathbb{N}\rightarrow \mathbb{N}$. We say that $um_k(G)$ is Lipschitz if
$|um_k(G)-um_k(G')|\leqslant 1$ for every $G$ and $G'$
which differ in at most one coordinate. We say that $um_k(G)$ is
$\theta$-certifiable if for any $G\in \mathcal{G}(n,p)$ and $\xi\in \mathbb{N}$
such that $um_k(G)\geqslant \xi$,
these exists a set of coordinates $I\subset [n]$ with $|I|\leqslant \theta(\xi)$ such that
each $G'\in \mathcal{G}(n,p)$ which agrees with $G$ on $I$ also satisfies $um_k(G')\geqslant \xi$.

For $k\geqslant 3$, $um_k(G)$ is not Lipschitz and certifiable.
If we change one coordinate of $G$ such that the chosen
vertex adjacent to every vertex of a largest $k$-matching in $G$,
 then the $k$-matching number in $G'$ possibly is  one. Hence,
the effect of changing one coordinate of $G$ is that the size of a considered $k$-matching
may drop from $\Theta( \frac{k n\log d}{d^{k-1}})$ to $1$,
which means that $um_k(G)$ is not Lipschitz. % significantly.
Similarly, $um_k(G)$ is also not $\theta$-certifiable for
$\theta(x)=2x$. %Tian~\cite{tian18}

Without Talagrand's inequality, we only
obtain a weak lower bound of $um_k(G)=\Omega(\log n)$ by the second moment method.
We show the proof in the Appendix for readers' reference. It's an improved version of~\cite{tian18}.
%Our proof also relies on the second moment method
% and Talagrand's inequality, while it is clearer and easier to understand than the one in~\cite{oliver21}.
\end{remark}

\section{Conclusions}

The main results of this paper characterize the value of $um_k(G)$ %by proving
%lower  and upper bounds
for some $G\in \mathcal{G}(n,p)$ and any fixed integer $k\geqslant 2$.
 The lower bounds here are within a factor two of the upper bound,
 where the first one is appropriate for any
maximal $k$-matching and the second one is obtained by a randomized greedy algorithm.
 The results partially generalize
some of the known results from the case when $k=2$ or $d=c$ for some large enough constant $c$.
It is interesting to consider the  lower bounds of
$um_k(G)$ for $G\in\mathcal{G}(n,p)$  when $k\geqslant 3$ and
$p$ lies in broader ranges. We also speculate the upper bound of $um_k(G)$ in Theorem~1.1 for some $G\in \mathcal{G}(n,p)$
is the asymptotic value of  $um_k(G)$. % as~\cite{kang12}.
The approach of Cooley et al.~\cite{oliver21}
only holds on $um_2(G)$.
We believe it is hard to improve our results for $k\geqslant 3$ by similar discussions.
Note that $d^{k-1}=o(n)$ is a necessary condition in the proof
of Lemma~2.4 to finally help us prove Theorem~1.1 and Theorem~1.3.
 This is the limitation of our approaches. It will take some new
ideas and these new problems will be more complicated than the one here.
% by the second moment method and Talagrand's inequality.
%It is interesting to consider the case when $k\geqslant 3$ and %the upper and lower bounds of
%$um_k(G)$  for $G\in\mathcal{G}(n,p)$ when
%$p$ lies in broader ranges.
%However, these bounds are far from
%the optimal.
%considered a lower bound on
%r-independence number in random hypergraphs.
We leave these problems for future work.

%\section*{Acknowledgements}
% Fang Tian
%was partially supported by the NSFC(No.~11871377,12071274). She
%is immensely grateful to the anonymous reviewers
%for their detailed and helpful
%suggestions, especially some suggestions in Section 2 and  Section 5. Without their help, it is impossible for us to write
%this manuscript in its current form. %on Section 5

\section*{Appendix}

%any given real number $\epsilon>0$ and
Without Talagrand's inequality, we show a weak lower bound
of $um_k(G)$ by the second moment method. For a fixed integer $k\geqslant 3$, define $m\cdot2^{m+1}=n$, which implies $m=\Theta(\log n)$.
Let $\mathcal{M}$ be the set of matchings of size $m$ in  $K_n$ and  $M_i\in \mathcal{M}$
for $1\leqslant i\leqslant t$ be the matchings in $\mathcal{M}$, where
\begin{align*}
t={n\choose 2m}{2m\choose
{2,\cdots,2}}\frac{1}{m!}.
\end{align*}
Let $I_i$ be the indicator
random variable of the event that
${M}_i$ is a $k$-matching in $G$ and $X_m=\sum_{i=1}^tI_i$.
%Note that $m\rightarrow\infty$ when $n\rightarrow\infty$ and $d^{k-1}=o(n)$.
We also have
\begin{align*}
\mathbb{E}[X_m]&= \binom{n}{2m} \binom{2m}{{2,\cdots,2}} \frac{1}{m!}\mathbb{P}\bigl[{M_i}\ {\rm is\ a\
} k{\mbox-}{\rm
matching}\bigr]\notag\\
&\sim \frac{n^{2m}}{2^mm!}\mathbb{P}\bigl[{M_i}\ {\rm is\ a\
} k{\mbox-}{\rm
matching}\bigr].
\tag{1}
\end{align*}
In order to  prove the lower bound of $um_k(G)$,
  by Chebyshev's inequality, we will show
$$
\mathbb{P}[um_k(G)<m]
\leqslant \mathbb{P}[X_m=0]
\leqslant\mathbb{P}\bigl[|X_m-\mathbb{E}(X_m)|\geqslant\mathbb{E}(X_m)\bigr]
\leqslant\frac{\mathbb{V}[{X_m}]}{\mathbb{E}^2(X_m)}
\rightarrow 0,
$$
then w.h.p. $um_k(G)\geqslant m$.

%\begin{proof}[Proof of Lower bound in Theorem~1.2]

%\vskip0.3cm

%\noindent{\bf Claim 4.1}\quad $\mathbb{E}(X_m)\rightarrow\infty$.

%\begin{proof} By the same argument in Section 3,
%it also follows that
%\begin{align}
%\mathbb{E}[X_m]
%&\sim \frac{1}{\sqrt{2\pi
%m}}\exp\Bigl[m\Bigl(\log\Bigl(\frac{ed n}{2m}\Bigr)-2(m-1)p_{d}+o(1)\Bigr)\Bigr]
%\end{align}
%when $m=\frac{(k-1)n\log d}{2d^{k-1}}$ in (4.1), while
%\begin{equation*}
%\begin{aligned}[b]
%\log\Bigl( \frac{ed n}{2m}\Bigr)&\sim k\log d-\log\log d,\\
% 2(m-1)p_{d}&\sim (k-1)\log d;
%\end{aligned}
%\end{equation*}
%and then
%\begin{align*}
%\log\Bigl( \frac{ed n}{2m}\Bigr)-2(m-1)p_{d}+o(1)\sim \log d-\log\log d>0
%\end{align*}
%such that $\mathbb{E}[X_m]\rightarrow \infty$.
%\end{proof}

%,  In order to show $\mathbb{P}[{M}_i\ {\rm and}\  {M}_j\
%{\rm are}\ k\text{-matchings in }G]$ for a fixed integer $k\geqslant 3$, i
Let ${M}_i,{M}_j\in\mathcal{M}$. It suffices to consider the case
where $M_i$ contains no edges with end-vertices lying in different edges of $M_j$ and vice verse.
Define
\begin{align*}
&C_e=\{e\in {M}_i\cap {M}_j\}\text{ with } c_e=|C_e|;\\
&C_v=\{v\in {M}_i\cap {M}_j\,|\, v\in e_i\in {M}_i, v\in e_j\in {M}_j, e_i\neq e_j \}\text{ with }
c_v=|C_v|.
\end{align*}
Namely, $C_e$ is the common edge set of ${M}_i$ and
${M}_j$, instead $C_v$ is the set of common vertices $v$ in
${M}_i$ and ${M}_j$ such that $v$ is incident to two
distinct edges, one in ${M}_i$ and the other in ${M}_j$.
Let ${{RM}_i}$ and ${{RM}_j}$
denote the rest edges in ${{M}_i}-C_e$ and
${{M}_j}-C_e$ that are not incident with $C_v$,
respectively. Denote
$|{{RM}_i}|=|{{RM}_j}|=m-c_e-c_v=r$, that is,  $c_e+c_v+r=m$ (see Fig.1).
The two matchings ${M}_i$ and ${M}_j$  characterized by
the above parameters are called  an $(r,c_v,c_e)$-pair.
Let $w_{r,c_v,c_e}$ be  the number of  $(r,c_v,c_e)$-pair
$({M}_i,{M}_j)$ for any given nonnegative
vector $(r,c_v,c_e)$ satisfying $r+c_e+c_v=m$.
Define
\begin{align*}
\mathcal{D}=\{(r,c_v,c_e)|\ r,c_v,c_e\ \text{are nonnegative integers such that } r+c_e+c_v=m\}.
\end{align*} %Using the equation
%shown in (4.4), we further have
For any given  $(r,c_v,c_e)\in \mathcal{D}$ and an $(r,c_v,c_e)$-pair
$({M}_i,{M}_j)$, let
\begin{align*}
\mathbb{E}_{r,c_v,c_e}=w_{r,c_v,c_e}\mathbb{P}\bigl[{M}_i\ {\rm and}\
{M}_j\ {\rm are}\ k\text{-matchings in }G\bigr].
\tag{2}
\end{align*}
Thus,
\begin{align*}
\mathbb{E}[X_m^2]=\sum_{(r,c_v,c_e)\in \mathcal{D}}\mathbb{E}_{r,c_v,c_e}.
\tag{3}
\end{align*}
\vskip 0.3cm
\noindent{\bf Claim A.1}\quad  For any given $(r,c_v,c_e)\in \mathcal{D}$, we have
\begin{align*}
w_{r,c_v,c_e}\sim \frac{n^{4m-(2c_e+c_v)}}{(r!)^2c_e!c_v!}2^{-2r-c_e}.
\end{align*}
%as $n\rightarrow \infty$.

\begin{proof}[Proof of Claim A.1] Since $w_{r,c_v,c_e}$ represents  the number of  $(r,c_v,c_e)$-pair
$({M}_i,{M}_j)$, thus
\begin{align*}
w_{r,c_v,c_e}&={n\choose {2r,2r,2c_e,c_v,c_v,c_v,n-4r-3c_v-2c_e}}\left[{2r\choose
{2,\cdots,2}}\frac{1}{r!}\right]^2\left[{2c_e\choose
{2,\cdots,2}}\frac{1}{c_e!}\right](c_v!)^2\\
&=\frac{n!}{(r!)^2c_e!c_v!(n-4r-3c_v-2c_e)!}2^{-2r-c_e}\\
&\sim\frac{n^{4m-(2c_e+c_v)}}{(r!)^2c_e!c_v!}2^{-2r-c_e},
\end{align*}
where the last approximate equality comes from the fact that $4r+3c_v+2c_e=4m-(2c_e+c_v)\leqslant 4m=o(n)$
 because  $m\cdot 2^{m+1}=n$.
\end{proof}
In particular, as $r=m$, then ${M}_i$ and
${M}_j$ are independent. By the equation in (1) and Claim A.1, we have
\begin{align*}
\mathbb{E}_{m,0,0}&=w_{m,0,0}\mathbb{P}[{M}_i\ {\rm and}\
{M}_j\ {\rm are}\ k\text{-matchings}]\\
&=w_{m,0,0}\mathbb{P}[{M}_i\ {\rm
is\ a}\ k\text{-matching}]^2\\
&\sim\mathbb{E}^2[X_m]
\end{align*}
and by the equation in (3),
\begin{align*}
\mathbb{V}[X_m]=\mathbb{E}[X_m^2]-\mathbb{E}^2[X_m]= \sum_{\small{(r,c_v,c_e)\in \mathcal{D},0\leqslant
r<m}}\mathbb{E}_{r,c_v,c_e}.
\tag{4}
\end{align*}

\begin{center}
\begin{tikzpicture}[scale=.4, vertex/.style = {shape = circle,fill=black, minimum size = 5pt, inner sep=0pt}]
\draw [draw=black, very thick] (7.6, 0.5) ellipse (3.5 cm and 2.8 cm);
\draw [draw=black, very thick] (3.6, 0.5) ellipse (3.5 cm and 2.8 cm);
\node at (1.3, 3.5) {\tiny ${M}_i$};
\node at (10.5, 3.5) {\tiny ${M}_j$};
\node at (8.5, 1.9) {\tiny ${RM}_j$};
\draw (7.5,1.2) [fill=black] circle (0.10);
\draw (7.5,-0.2) [fill=black] circle (0.10);
\draw (8.3,1.2) [fill=black] circle (0.10);
\draw (8.3,-0.2) [fill=black] circle (0.10);
\draw (9.1,1.2) [fill=black] circle (0.10);
\draw (9.1,-0.2) [fill=black] circle (0.10);
\node at (7.9, 0.5) {$\cdots$};
\node at (8.7, 0.5) {$\cdots$};

\draw (5.1,1.5) [fill=black] circle (0.10);
\draw (5.1,0.1) [fill=black] circle (0.10);
\draw (6.0,1.5) [fill=black] circle (0.10);
\draw (6.0,0.1) [fill=black] circle (0.10);
\node at (5.6, 0.8) {$\cdots$};
\node at (5.6, 2.1) {{\tiny $C_e$}};

\draw (2.2,1.2) [fill=black] circle (0.10);
\draw (2.2,-0.2) [fill=black] circle (0.10);
\draw (3.0,1.2) [fill=black] circle (0.10);
\draw (3.0,-0.2) [fill=black] circle (0.10);
\draw (3.8,1.2) [fill=black] circle (0.10);
\draw (3.8,-0.2) [fill=black] circle (0.10);
\node at (2.6, 0.5) {$\cdots$};
\node at (3.4, 0.5) {$\cdots$};
\node at (3, 2) {{\tiny ${RM}_i$}};

\draw (5.8,-0.5) [fill=black] circle (0.10);
\draw (7.3,-0.7) [fill=black] circle (0.10);
\draw [>=stealth, shorten >=2pt,line width=1pt] (5.8,-0.5)--(7.3,-0.7);
\node at (5.6, -1.1) {{\tiny$C_v$}};

\draw (3.3,-0.7) [fill=black] circle (0.10);
\draw [>=stealth, shorten >=2pt,line width=1pt] (5.8,-0.5)--(3.3,-0.7);
\draw [>=stealth, shorten >=2pt,line width=1pt] (7.5,1.2)--(7.5,-0.2);
\draw [>=stealth, shorten >=2pt,line width=1pt] (8.3,1.2)--(8.3,-0.2);
\draw [>=stealth, shorten >=2pt,line width=1pt] (9.1,1.2)--(9.1,-0.2);
\draw [>=stealth, shorten >=2pt,line width=1pt] (5.1,1.5)--(5.1,0.1);
\draw [>=stealth, shorten >=2pt,line width=1pt] (6.0,1.5)--(6.0,0.1);
\draw [>=stealth, shorten >=2pt,line width=1pt] (2.2,1.2)--(2.2,-0.2);
\draw [>=stealth, shorten >=2pt,line width=1pt] (3.0,1.2)--(3.0,-0.2);
\draw [>=stealth, shorten >=2pt,line width=1pt] (3.8,1.2)--(3.8,-0.2);
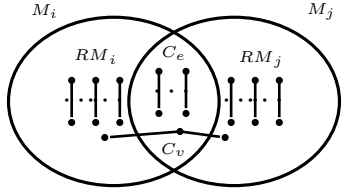
\node at (6.5, -3.7) {\small {Fig.1\quad The $(r,c_v,c_e)$-pair ${M}_i$ and ${M}_j$.}};
\end{tikzpicture}
\end{center}

\vskip 0.3cm

\noindent{\bf Claim A.2}\quad For any given $(r,c_v,c_e)\in \mathcal{D}$ with $0\leqslant r<m$, let ${M}_i$ and
${M}_j$ be an $(r,c_v,c_e)$-pair. Then,
%\noindent(a) If $r=m$, then
%$\frac{\mathbb{P}[{M}_i,{M}_j \ {\rm are\ both}\
%k\text{-matchings}]}{\mathbb{P}[{M}_i\ {\rm
%is\ a}\ k-{\rm matching}]^2}=1$.\noindent(b)
%If $0\leqslant r<m$, then
\begin{align*}
\frac{\mathbb{P}[{M}_i,{M}_j \ {\rm are\ both}\
k\text{-matchings}]}{\mathbb{P}[{M}_i\ {\rm
is\ a}\ k\text{-\rm matching}]^2}\sim
p^{-c_e}(1-p_d)^{\frac{1}{2}[4c_e+c_v-(2c_e+c_v)^2]}.
\end{align*}

\begin{proof}[Proof of Claim A.2]
% As $r=m$, (a) is
%trivial because  ${M}_i$ and
%${M}_j$ are independent.
%For (b),
Let ${M}_i$ and ${M}_j$ be an $(r,c_v,c_e)$-pair. We have known
\begin{align*}
\mathbb{P}[{M}_i\ {\rm is\ a}\
k\text{-matching}]
=p^m\mathbb{P}[d_{G}(u,v)\geqslant k]^{2m(m-1)},
\tag{5}
\end{align*} thus
it suffices to consider $\mathbb{P}[{M}_i,{M}_j \ {\rm are\ both}\
k\text{-matchings}]$. First, there exists a penalty factor $p^{2m-c_e}$ to guarantee the
edges of ${M}_i\cup {M}_j$ present in $G$, and
the total number of the paths which join vertices on any two edges
in ${M}_i$ or any two ones in ${M}_j$ is exactly
\begin{align*}
&2{2m \choose 2}-{2c_e+c_v\choose 2}-(2m-c_e)\\
&=8{m\choose 2}-{2c_e+c_v\choose 2}+c_e\\
&= 4m^2-4m+ \frac{1}{2}\bigl[4c_e+c_v-(2c_e+c_v)^2\bigr],
\end{align*}
where $2{2m \choose 2}-{2c_e+c_v\choose 2}$ is the sum of the number of pairs of vertices in
${M}_i$ and the number of pairs of vertices in
${M}_j$, and $(2m-c_e)$ is the number of edges in ${M}_i\cup {M}_j$.
Thus,
\begin{align*}
&\mathbb{P}[{M}_i,{M}_j\ {\rm are\ both}\
k\text{-matchings}]\notag\\
&=p^{2m-c_e}\mathbb{P}[d_{G}(u,v)
\geqslant
k]^{4m^2-4m+\frac{1}{2}[4c_e+c_v-(2c_e+c_v)^2]}.
\tag{6}
\end{align*}Using the equations in (5) and (6),
\begin{align*}
&\frac{\mathbb{P}[{M}_i,{M}_j \ {\rm are\ both}\
k\text{-matchings}]}{\mathbb{P}[{M}_i\ {\rm
is\ a}\ k\text{-matching}]^2}\notag\\&=
p^{-c_e}\mathbb{P}[d_{G}(u,v)\geqslant
k]^{\frac{1}{2}[4c_e+c_v-(2c_e+c_v)^2]}\notag\\
&\sim p^{-c_e}(1-p_d)^{\frac{1}{2}[4c_e+c_v-(2c_e+c_v)^2]},
\end{align*}
where the last approximate equality comes from
$\mathbb{P}[d_{G}(u,v)\geqslant
k]\sim 1-p_d$ when $n\rightarrow \infty$ because
\begin{align*}
\exp\bigl[O(n^{k-3}p^{k-2}+n^{2k-4}p^{2k-2})\bigr]\rightarrow 1
\end{align*}
in Lemma 2.4\,(a) when $d^{k-1}=o(n)$. %\hfill$\square$
Hence, Claim A.2 holds.
\end{proof}

\vskip 0.3cm

Using the equations in (1), (2), Claim A.1 and Claim A.2, for $0\leqslant r< m$,
we also have
\begin{align*}
\frac{\mathbb{E}_{r,c_v,c_e}}{\mathbb{E}^2[X_m]}
&\sim\frac{(m!)^2n^{-2c_e-c_v}}{(r!)^2c_e!c_v!}2^{2c_v+c_e}p^{-c_e}(1-p_d)^{\frac{1}{2}[4c_e+c_v-(2c_e+c_v)^2]}.
\tag{7}
\end{align*}

\noindent{\bf Claim A.3}\quad Let
$
f(x)=\bigl(1-\frac{2n}{d^{k-1}}\log\frac{2m}{n}\bigr)x-4x^2,
$
then
$
f(x)\geqslant
f(m)\sim\frac{n(4m+2)}{c^{k-1}}\log 2
$ for $1\leqslant x\leqslant m$.

\begin{proof}[Proof of Claim A.3] Since $f(x)$ is a symmetric concave downward
quadratic function, it is easy to see the maximum point of $f(x)$ occurs at %analyze that $f(x)$ is first
%increasing then decreasing, and
$$
x=\frac{1-\frac{2n}{d^{k-1}}\log\frac{2m}{n}}{8}
\sim\frac{mn\log 2}{4d^{k-1}}>m
$$
when $m\cdot2^{m+1}=n$.
Therefore,
$$
f(x)\geqslant
f(1)= \frac{2mn\log 2}{d^{k-1}}-3\sim \frac{2mn\log 2}{d^{k-1}}
$$ when $1\leqslant x\leqslant m$.
\end{proof}

\vskip 0.3cm
\noindent{\bf Claim A.4}\quad $\mathbb{V}[{X_m}]=o(\mathbb{E}^2[X_m])$.

\begin{proof}[Proof of Claim A.4] By the equation in (4), in order to show $\mathbb{V}[{X_m}]=o(\mathbb{E}^2[X_m])$, it remains to prove that
%as $n\rightarrow\infty$,
$$
\sum_{(r,c_v,c_e), 0\leqslant
r<m}\frac{\mathbb{E}_{r,c_v,c_e}}{\mathbb{E}^2[X_m]}\rightarrow 0. %\rightarrow 0\qquad{\rm as}\ n\rightarrow\infty,
%\sim&\sum\limits_{0\leq
%r<m,r+c_v+c_e=m}\frac{(m!)^2n^{-2c_e-c_v}}{(r!)^2c_e!c_v!}2^{2c_v+c_e}p^{-c_e}(1-p_c)^{\frac{1}{2}[4c_e+c_v-(2c_e+c_v)^2]}\\
%=&\sum\limits_{0\leq
%r<m,r+c_v+c_e=m}\frac{(m!)^2}{(r!)^2c_e!c_v!}4^{c_v}\left(\frac{2(1-p_c)}{p}\right)^{c_e}(1-p_c)^{f(2c_e+c_v)}\\
$$
%By (2) and Claim A.2, we have
%$$
%\mathbb{E}_{r,c_v,c_e}\sim\frac{n^{4m-(2c_e+c_v)}}{(r!)^2c_e!c_v!}2^{-2r-c_e}\mathbb{P}[{M}_i\ {\rm and}\
%{M}_j\ {\rm are}\ k\text{-matchings}].
%$$ and Claim 1.2\,(b)
By the equation in (7) and $m=r+c_e+c_v$, we have
\begin{align*}
\frac{\mathbb{E}_{r,c_v,c_e}}{\mathbb{E}^2(X_m)} %\rightarrow 0\qquad{\rm as}\ n\rightarrow\infty,
&\sim\frac{(m!)^2n^{-2c_e-c_v}}{(r!)^2c_e!c_v!}2^{2c_v+c_e}p^{-c_e}(1-p_d)^{\frac{1}{2}[4c_e+c_v-(2c_e+c_v)^2]}\\
&<\frac{m!}{r!c_e!c_v!}2^{c_v}\Bigl(\frac{(1-p_d)^{1.5}}{d}\Bigr)^{c_e}\Bigl(\frac{2m}{n}\Bigr)^{c_e+c_v}(1-p_d)^{\frac{1}{2}[c_e+c_v-4(c_e+c_v)^2]},
\tag{8}
\end{align*}
where the inequality comes from $\frac{m!}{r!}<m^{m-r}=m^{c_v+c_e}$, $d=np$ and
$$
(1-p_d)^{\frac{1}{2}[4c_e+c_v-(2c_e+c_v)^2]}<(1-p_d)^{\frac{1}{2}[4c_e+c_v-4(c_e+c_v)^2]}.
$$

Note that
\begin{align*}
&\Bigl(\frac{2m}{n}\Bigr)^{c_e+c_v}(1-p_d)^{\frac{1}{2}[c_e+c_v-4(c_e+c_v)^2]}\\
&=(1-p_d)^{\frac{1}{2}(c_e+c_v)\left[1+2\log_{1-p_d}\frac{2m}{n}-4(c_e+c_v)\right]}\\
&=(1-p_d)^{\frac{1}{2}(c_e+c_v)\big[1+2\frac{\log\frac{2m}{n}}{\log(1-p_d)}-4(c_e+c_v)\big]}\\
&\sim(1-p_d)^{\frac{1}{2}(c_e+c_v)\left[1-\frac{2n}{d^{k-1}}\log\frac{2m}{n}-4(c_e+c_v)\right]},
\end{align*}
where the last approximate equality is true because  $\log(1-p_d)\sim -p_d$ when $p_d=o(1)$.
By Claim A.3 and $1\leqslant c_e+c_v=m-r\leqslant m$ when $0\leqslant r< m$,
we further have
\begin{align*}
\frac{\mathbb{E}_{r,c_v,c_e}}{\mathbb{E}^2(X_m)}
&<\frac{m!}{r!c_e!c_v!}2^{c_v}\Bigl(\frac{(1-p_d)^{1.5}}{d}\Bigr)^{c_e}(1-p_d)^{\frac{2mn}{d^{k-1}}\log
2}.
%\quad{\rm{(by\ Lemma\ 3.2.)}}
%\tag{8}
\end{align*}
%where the last two inequalities come from Claim 1.5 because we take $x=c_e+c_v=m-r\geq 1$.

%By (2) in Lemma 3.1, $p=\frac{c}{n}$ and $\log(1-p_c)\sim -p_c$ as $p_c\rightarrow 0$, we have
At last, it follows that
\begin{align*}
\sum_{(r,c_v,c_e), 0\leqslant
r<m}\frac{\mathbb{E}_{r,c_v,c_e}}{\mathbb{E}^2(X_m)}&<
(1-p_d)^{{\frac{2mn}{d^{k-1}}\log
2}}\sum_{(r,c_v,c_e), 0\leqslant
r<m}\frac{m!}{r!c_e!c_v!}2^{c_v}\Bigl(\frac{(1-p_d)^{1.5}}{d}\Bigr)^{c_e}\\
& <(1-p_d)^{{\frac{2mn}{c^{k-1}}\log
2}}\Bigl(3+\frac{(1-p_d)^{1.5}}{d}\Bigr)^m\\
&=\Bigl[(1-p_d)^{\frac{2n}{c^{k-1}}\log
2}\Bigl(3+\frac{(1-p_d)^{1.5}}{d}\Bigr)\Bigr]^m,%\\\rightarrow& 0,
\tag{9}\end{align*}
where the second inequality uses the fact that for real number
$x_0,\cdots,x_l$, %it holds that
$$
(x_0+x_1+\cdots+x_l)^m=\sum_{\begin{subarray}{l}k_0,\cdots,k_l\geqslant
0\\k_0+\cdots+k_l=m\end{subarray}}\frac{m!}{k_0!\cdots k_l!}x_0^{k_0}\cdots x_l^{k_l}.
%(x_0+x_1+\cdots+x_l)^m=\sum_{\begin{subarray}{l}k_0,\cdots,k_l\geq
%0\\k_0+\cdots+k_l=m\end{subarray}}\frac{m!}{k_0!\cdots k_l!}x_0^{k_0}\cdots x_l^{k_l},
$$
In (9), since $p_d=d^{k-1}/n=o(1)$, we have
\begin{align*}
&(1-p_d)^{\frac{2n}{d^{k-1}}\log
2}\Bigl(3+\frac{(1-p_d)^{1.5}}{d}\Bigr)
\sim 3(1-p_d)^{\frac{2n}{c^{k-1}}\log
2}\rightarrow \frac{3}{4},
\end{align*}
thus
$
\sum_{(r,c_v,c_e), 0\leqslant
r<m}\frac{\mathbb{E}_{r,c_v,c_e}}{\mathbb{E}^2(X_m)}\rightarrow 0 %\rightarrow 0\qquad{\rm as}\ n\rightarrow\infty,
%\sim&\sum\limits_{0\leq
%r<m,r+c_v+c_e=m}\frac{(m!)^2n^{-2c_e-c_v}}{(r!)^2c_e!c_v!}2^{2c_v+c_e}p^{-c_e}(1-p_c)^{\frac{1}{2}[4c_e+c_v-(2c_e+c_v)^2]}\\
%=&\sum\limits_{0\leq
%r<m,r+c_v+c_e=m}\frac{(m!)^2}{(r!)^2c_e!c_v!}4^{c_v}\left(\frac{2(1-p_c)}{p}\right)^{c_e}(1-p_c)^{f(2c_e+c_v)}\\
$ and Claim A.4 holds.
\end{proof}

%\end{proof}

\end{document}